\documentclass[a4paper, reqno, 14pt]{amsart}

\usepackage[usenames,dvipsnames]{color}
\usepackage{amsthm,amsfonts,amssymb,amsmath,amsxtra}
\usepackage[all]{xy}
\SelectTips{cm}{}
\usepackage{xr-hyper}
\usepackage[colorlinks=
citecolor=Black,
linkcolor=Red,
urlcolor=Blue]{hyperref}

\usepackage{verbatim}

\usepackage[margin=1.25in]{geometry}

\usepackage{mathrsfs}

\usepackage{diagbox}

\usepackage{tikz}

\usepackage{tkz-euclide}

\usepackage{setspace} % 导入setspace宏包

\RequirePackage{xspace}
\RequirePackage{etoolbox}
\RequirePackage{varwidth}
\RequirePackage{enumitem}
\RequirePackage{tensor}
\RequirePackage{mathtools}
\RequirePackage{longtable}
\RequirePackage{multirow}

\usepackage{tikz}
\usepackage{tikz-cd}
\usepackage{csquotes}

\usepackage[url=false,doi=false,isbn=false,sorting=nyt,giveninits=true,maxbibnames=4]{biblatex}
\bibliography{references.bib}

\def\ge{\geqslant}
\def\le{\leqslant}
\def\a{\alpha}
\def\b{\beta}
\def\g{\gamma}
\def\G{\Gamma}
\def\d{\delta}

\def\o{\omega}

\def\s{\sigma}
\def\t{\tau}

\def\k{\kappa}
\def\l{\lambda}
\def\z{\zeta}

\def\i{^{-1}}

\def\<{\langle}
\def\>{\rangle}

\newcommand{\bJ}{\mathbf J}
\newcommand{\de}{{\mathrm{def}}}

\newcommand{\bH}{\mathbf H}
\newcommand{\bG}{\mathbf G}

\newcommand{{\BG}}{\ensuremath{\mathbb {G}}\xspace}

\newcommand{{\BK}}{\ensuremath{\mathbb {K}}\xspace}

\newcommand{\BN}{\ensuremath{\mathbb {N}}\xspace}

\newcommand{\BQ}{\ensuremath{\mathbb {Q}}\xspace}
\newcommand{\BR}{\ensuremath{\mathbb {R}}\xspace}
\newcommand{\BS}{\ensuremath{\mathbb {S}}\xspace}

\newcommand{\BZ}{\ensuremath{\mathbb {Z}}\xspace}

\newcommand{\CA}{\ensuremath{\mathcal {A}}\xspace}

\newcommand{\CI}{\ensuremath{\mathcal {I}}\xspace}

\newcommand{\CK}{\ensuremath{\mathcal {K}}\xspace}

\newcommand{\Ad}{{\mathrm{Ad}}}
\newcommand{\ad}{{\mathrm{ad}}}
\newcommand{\indec}{{\mathrm{indec}}}

\newcommand{\leng}{{\mathrm{length}}}

\DeclareMathOperator{\Adm}{Adm}

\newcommand{\GL}{\mathrm{GL}}

\newcommand{\id}{\ensuremath{\mathrm{id}}\xspace}

\def\tW{\tilde W}
\def\tS{\tilde \BS}

\newtheorem{theorem}{Theorem}
\newtheorem{proposition}[theorem]{Proposition}
\newtheorem{lemma}[theorem]{Lemma}

\newtheorem{corollary}[theorem]{Corollary}

\theoremstyle{definition}
\newtheorem{definition}[theorem]{Definition}
\newtheorem{example}[theorem]{Example}
\newtheorem*{example*}{Example}

\newtheorem{remark}[theorem]{Remark}

\newtheorem*{function*}{Function}

\numberwithin{equation}{section}
\numberwithin{theorem}{section}

\setitemize[0]{leftmargin=*,itemsep=\the\smallskipamount}
\setenumerate[0]{leftmargin=*,itemsep=\the\smallskipamount}

\renewcommand{\to}{%
   \ifbool{@display}{\longrightarrow}{\rightarrow}%
   }
\let\shortmapsto\mapsto
\renewcommand{\mapsto}{%
   \ifbool{@display}{\longmapsto}{\shortmapsto}%
   }
\newlength{\olen}
\newlength{\ulen}
\newlength{\xlen}
\newcommand{\xra}[2][]{%
   \ifbool{@display}%
      {\settowidth{\olen}{$\overset{#2}{\longrightarrow}$}%
       \settowidth{\ulen}{$\underset{#1}{\longrightarrow}$}%
       \settowidth{\xlen}{$\xrightarrow[#1]{#2}$}%
       \ifdimgreater{\olen}{\xlen}%
          {\underset{#1}{\overset{#2}{\longrightarrow}}}%
          {\ifdimgreater{\ulen}{\xlen}%
             {\underset{#1}{\overset{#2}{\longrightarrow}}}
             {\xrightarrow[#1]{#2}}}}%
      {\xrightarrow[#1]{#2}}
   }
\makeatother
\newcommand{\xyra}[2][]{%
   \settowidth{\xlen}{$\xrightarrow[#1]{#2}$}%
   \ifbool{@display}%
      {\settowidth{\olen}{$\overset{#2}{\longrightarrow}$}%
       \settowidth{\ulen}{$\underset{#1}{\longrightarrow}$}%
       \ifdimgreater{\olen}{\xlen}%
          {\mathrel{\xymatrix@M=.12ex@C=3.2ex{\ar[r]^-{#2}_-{#1} &}}}%
          {\ifdimgreater{\ulen}{\xlen}%
             {\mathrel{\xymatrix@M=.12ex@C=3.2ex{\ar[r]^-{#2}_-{#1} &}}}
             {\mathrel{\xymatrix@M=.12ex@C=\the\xlen{\ar[r]^-{#2}_-{#1} &}}}}}%
      {\mathrel{\xymatrix@M=.12ex@C=\the\xlen{\ar[r]^-{#2}_-{#1} &}}}%
   }
\makeatletter
\newcommand{\xla}[2][]{%
   \ifbool{@display}%
      {\settowidth{\olen}{$\overset{#2}{\longleftarrow}$}%
       \settowidth{\ulen}{$\underset{#1}{\longleftarrow}$}%
       \settowidth{\xlen}{$\xleftarrow[#1]{#2}$}%
       \ifdimgreater{\olen}{\xlen}%
          {\underset{#1}{\overset{#2}{\longleftarrow}}}%
          {\ifdimgreater{\ulen}{\xlen}%
             {\underset{#1}{\overset{#2}{\longleftarrow}}}
             {\xleftarrow[#1]{#2}}}}%
      {\xleftarrow[#1]{#2}}
   }
\newcommand{\isoarrow}{%
   \ifbool{@display}{\overset{\sim}{\longrightarrow}}{\xrightarrow\sim}%
   }

\begin{document}

\title[]{Zero-dimensional affine Deligne--Lusztig varieties} 
\author[Xuhua He]{Xuhua He}
\address{Department of Mathematics and New Cornerstone Science Laboratory, The University of Hong Kong, Pokfulam, Hong Kong, Hong Kong SAR, China}
\email{xuhuahe@hku.hk}

\author[Sian Nie]{Sian Nie}
\address{Academy of Mathematics and Systems Science, Chinese Academy of Sciences, Beijing 100190, China, and, School of Mathematical Sciences, University of Chinese Academy of Sciences, Chinese Academy of Sciences, Beijing 100049, China}
\email{niesian@amss.ac.cn}

\author[Qingchao Yu]{Qingchao Yu}
\address{Department of Mathematics and New Cornerstone Science Laboratory, The University of Hong Kong, Pokfulam, Hong Kong, Hong Kong SAR, China}
\email{qingchao\_yu@outlook.com}
\thanks{}

\keywords{Affine Deligne--Lusztig varieties}
\subjclass[2010]{11G25, 20G25}

%\date{\today}

\begin{abstract}
In this paper, we study the affine Deligne--Lusztig variety $X(\mu,b)_K$ and classify all quadruples $(\bG, \mu, b, K)$ with $\dim X(\mu, b)_K=0$. This question was first asked by Rapoport in 2005, who also made an explicit conjecture in the hyperspecial level. We prove that $\dim X(\mu,b)_K=0$ if and only if, up to certain Hodge-Newton decomposition condition, the pair $(\bG, \{\mu\})$ is of extended Lubin-Tate type. We also give a combinatorial description of this condition by the essential gap function on $B(\bG)$ and the $\mu$-ordinary condition for the generic Newton stratum.

\end{abstract}

\maketitle

%\tableofcontents

\section*{$Introduction$}

\subsection{Motivation} The notion of affine Deligne--Lusztig variety was introduced by Rapoport in \cite{Ra}. It plays an important role in arithmetic geometry and the Langlands program. In this paper, we consider the affine Deligne--Lusztig variety
$$X(\mu, b)_K  =\{g \breve \CK\mid g \i b \s(g) \in \breve \CK \Adm(\{\mu\}) \breve \CK\}.$$
We refer to \S\ref{sec:3.2} for the notation used here. It is known (see \cite{He16}) that $X(\mu, b)_K \neq \emptyset$ if and only if the $\s$-conjugacy class of $b$ lies in the neutral acceptable set $B(\bG, \{\mu\})$. It is a perfect closed subscheme of the partial affine flag variety $\breve G/\breve \CK$ in the sense of \cite{BS} and \cite{Zhu}.
%Here $\bG$ is a connected reductive group over a non-archimedean local field $F$, $\breve G=\bG(\breve F)$, where $\breve F$ is the completion of the maximal unramified extension of $F$, and $\s$ is the Frobenius morphism on $\breve G$. The set $K$ denotes an $F$-rational parahoric structure of $\bG$ and $\breve \CK$ is the corresponding standard parahoric subgroup of $\breve G$, $b \in \breve G$, $\mu$ is a conjugacy class of cocharacters of $\bG$ and $\Adm(\{\mu\})$ is the $\mu$-admissible subset of the Iwahori--Weyl group of $\breve G$. 

%The set $X(\breve G, \{\mu\})$ has a natural geometric structure. In the equal characteristic case, $X(\mu, b)_K$ is a finite dimensional closed subscheme of the partial affine flag variety $\breve G/\breve \CK$, locally of finite type over the residue field of $\breve F$. In the mixed characteristic case, both the partial flag variety and the subset $X(\mu, b)_K$ in $\breve G/\breve \CK$ are understood in the sense of \cite{BS} and \cite{Zhu} as perfect schemes.

In the mixed characteristic case, the set $X (\mu, b)_K$, associated with a Shimura datum $(\bG, \{\mu\})$, arises as the set of geometric points of the underlying reduced scheme of a Rapoport-Zink formal moduli space of $p$-divisible groups (see \cite[\S 4]{RV}). In the equal characteristic case, the set $X(\mu, b)_K$ arises from the formal moduli space of Shtukas (see \cite{Vi18}). In this case, there is no restriction on the pair $(\bG, \{\mu\})$. 

%\subsection{Previous works}\label{sec:pre-work} 
%The isomorphism classes of the affine Deligne--Lusztig variety $X(\mu, b)$ depends on the $\s$-conjugacy class $[b]$ of $b$ in $\breve G$. The first question about $X(\mu, b)_K$ is to determine its non-emptiness pattern. Kottwitz and Rapoport conjectured that $X(\mu, b)_K \neq \emptyset$ if and only if $[b]$ is in the neutral acceptable set $B(\bG,\{\mu\})$. This conjecture was established by the first author in \cite{He16}. 

The set $B(\bG,\{\mu\})$ admits a natural partial order. It has a unique minimal element, the basic $\s$-conjugacy class in $B(\bG,\{\mu\})$. It is shown in \cite{HN2018} that $B(\bG,\{\mu\})$ also has a unique maximal element. When $(\bG, \{\mu\})$ is a Shimura datum, $X(\mu, b)_K$ for basic $b$ (resp. maximal $b$) is the group-theoretic model of the basic locus (resp. maximal locus) of the corresponding Shimura variety with parahoric level $K$. When $\bG$ is quasi-split, the maximal locus is the $\mu$-ordinary locus.

The formal moduli spaces of $p$-divisible groups and Shtukas, in general, are very complicated. However, it has been known for decades that in the Lubin-Tate case, the underlying reduced scheme (i.e. $X (\mu, b)_K$) is zero-dimensional. 
%the explicit geometric structures are known: the Lubin-Tate case and the Drinfeld case. In the Lubin-Tate case, the formal scheme is a disjoint union of formal spectra of formal power series rings, and the underlying reduced scheme is a discrete set (and hence is zero dimensional). Besides the Lubin-Tate case, there is another family of the quadruple $(\bG, \mu, b, K)$ such that the associated affine Deligne--Lusztig variety $X^{\bG}(\mu, b)_K$ is zero-dimensional: the case of $\mu$-ordinary locus for quasi-split groups. 

It is a natural and interesting question to classify all quadruples $(\bG, \mu, b, K)$ such that $\dim X (\mu, b)_K=0$. This question was first asked by Rapoport in 2005, who also made an explicit conjecture in the hyperspecial level. See \cite[Page 5]{Vi05}. Rapoport's conjecture was verified for $\bG=\GL_n$ by Viehmann in loc.cit. For other groups, it follows from the explicit dimension formula of affine Deligne--Lusztig varieties with hyperspecial level structures, established for split groups by G\"ortz, Haines, Kottwitz, Reuman in \cite{GHKR06} and Viehmann in \cite{Vi06}, and for quasi-split groups by Hamacher in \cite{Ham15b} and by X. Zhu in \cite{Zhu}.

For other parahoric level structure, it is a challenging problem to determine the dimension of $X (\mu, b)_K$. We refer to \cite{He-ICM} and \cite{HY} for the discussion of the results in this direction. Note that in most known cases, there is an essential ``gap'' between $\mu$ and $b$ and the associated affine Deligne--Lusztig variety is not zero dimensional. In particular, the known dimension formulas are not very useful in the classification problem for $\dim X (\mu, b)_K=0$.

The recent work by G\"{o}rtz, Rapoport and the first author \cite{GHR} classifies the cases that $\dim X(\mu, b)_K = 0$ when $b$ is a basic element. It is shown that when $b$ is basic, $X (\mu, b)_K$ is 0-dimensional if and only if $(\bG, \{\mu\})$ is the extended Lubin-Tate case. The key idea is that when $(\bG, \{\mu\})$ is not the Lubin-Tate case, then one can construct an explicit curve in $X(\mu, b)_K$. This construction relies on the condition that $b$ is basic. A priori, it is more difficult to construct an explicit curve in $X(\mu, b)_K$ when $[b]$ is ``close'' to $\mu$.

\subsection{Main result} The main purpose of this paper is to give a classification of the cases $\dim X(\mu, b)_K = 0$. We state our main result. 

\begin{theorem} \label{main-thm}
The following conditions are equivalent

(1) $\dim X(\mu, b)_K = 0$;

(2) The datum $(\bG, \{\mu\})$ is of extended Lubin-Tate type after applying the Hodge-Newton decomposition;

(3) The maximum element $[b_{\mu,\max}]$ of $B(\bG,\{\mu\})$ is $\mu$-ordinary and the essential gap between $[b]$ and $[b_{\mu,\max}]$ is zero.
\end{theorem}

Note that the $\mu$-ordinary $\s$-conjugacy class, if exists, must be the maximal element $[b_{\mu, \max}]$ of $B(\bG,\{\mu\})$. However, in general $[b_{\mu, \max}]$ may not be $\mu$-ordinary. It is known that for quasi-split groups, $[b_{\mu, \max}]$ is $\mu$-ordinary. We refer to \S\ref{sec:mu-ordinary} for the classification of pairs $(\bG, \{\mu\})$ for the non quasi-split groups such that $[b_{\mu, \max}]$ is $\mu$-ordinary.

The key ingredients of the proof consist of: 

\begin{itemize}
    \item Hodge-Newton decomposition of $X(\mu, b)_K$;  
    \item Purity theorem; 
    \item Essential gap function on $B(\bG )$. 
\end{itemize}

The Hodge-Newton decomposition for $X(\mu, b)_K$ was first established in the hyperspecial case by Kottwitz \cite{K03}. For the arbitrary parahoric level, the Hodge-Newton decomposition was proved in \cite{GHN2}. Roughly speaking, when $(\mu, [b])$ is Hodge-Newton decomposable, the Hodge-Newton decomposition gives an isomorphism for the affine Deligne--Lusztig variety $X(\mu, b)$ for $\bG$ to a disjoint union of affine Deligne--Lusztig varieties\footnote{For hyperspecial level, only one Levi subgroup is involved. However, several Levi subgroups are needed in other cases of level structure.} for various semi-standard Levi subgroups of $\bG$. In particular, the Hodge-Newton decomposition gives the implication (2) $\Rightarrow$ (1).

The purity theorem was first established by de Jong and Oort \cite{dJO} for $F$-isocrystals. Namely, it says that the complement $\overline S - S$ of a Newton stratum $S$ is pure of codimension $\le 1$ in the closure $\overline S$. The purity theorem was later strengthened and extended to $F$-isocrystals with $G$-structure by Vasiu \cite{Va}, Yang \cite{Ya}, Viehmann \cite{Vi13}, and Hamacher \cite{Ham17}. Applying the purity theorem, we obtain in Proposition \ref{prop:dimxmub} an explicit lower bound for $\dim X(\mu, b)$. In particular, if $\dim X(\mu, b) = 0$, then condition (3) of Theorem \ref{main-thm} holds.
Finally, by a purely combinatorial argument, we obtain the implication (3) $\Rightarrow$ (2). 

In this paper, we introduce the essential gaps among elements in $B(\bG)$. This is a key ingredient in our proof and is also of independent interest. It is known that $B(\bG)$ is a ranked poset. Thus, we may estimate the distance between $[b]$ and $[b']$ in $B(\bG )$ with $[b] \le [b']$ by the length of some/any maximal chain from $[b]$ to $[b']$. On the other hand, to each $[b] \in B(\bG )$ we associate its Newton vector $\nu_b$, which is a rational dominant coweight. We may define another distance function between $[b]$ and $[b']$ as $\<\nu_{b'} - \nu_b, 2\rho\>$, where $\rho$ is the half sum of positive roots. It is easy to see that $\<\nu_{b'} - \nu_b, 2\rho\> \ge \leng_{\bG}([b], [b'])$. We define the essential gap between $[b]$ and $[b']$ as \[\<\nu_{b'} - \nu_b, 2\rho\> - \leng_{\bG}([b], [b']).\]

Finally, we would like to point out the following consequences 
\begin{itemize}
    \item The classification of cases $\dim X(\mu, b)_K = 0$ is independent of the choice of the parahoric level $K$.
    
    \item If $\dim X(\mu, b)_K = 0$, then $\dim X(\mu, b')_K = 0$ for all $[b'] \in B(\bG,\{\mu\})$ such that $[b] \le [b']$.
\end{itemize}

\smallskip

\noindent {\bf Acknowledgment: } The use of Pick's Theorem in \S\ref{sec:2.2} is partially motivated by the work of Lim \cite{Lim}. We thank Michael Rapoport for his valuable comments on an earlier version of this paper. We also thank Felix Schremmer for helpful discussions and comments. XH is partially supported by the New Cornerstone Science Foundation through the New Cornerstone Investigator Program and the Xplorer Prize, and by Hong Kong RGC grant 14300220. 

\section{Preliminary}

\subsection{The reductive group.}\label{sec:1.1}
Let $F$ be a non-archimedean local field and let $\bG$ be a connected reductive group over $F$. We set $\breve G=\bG(\breve F)$, where $\breve F$ is the completion of the maximal unramified extension $F^{\mathrm{un}}$ of $F$. Let $\G$ be the Galois group of $\bar{F}$ over $F$ and $\G_0$ be the Galois group of $\bar{F}$ over $F^{\mathrm{un}}$. Let $\s$ be the Frobenius morphism. 

Let $S$ be a maximal $\breve F$-split torus of $\bG$ defined over $F$. Let $T$ be the centralizer of $S$ in $\bG$. By Steinberg’s theorem, $\bG$ is quasi-split over $\breve F$. Then $T$ is a maximal torus. We denote by $N_T$ the normalizer of $T$. Let $W=N_T(\breve F)/T(\breve F)$ be the (relative) Weyl group. Let $X_*(T)$ be the cocharacter group of $T$ and let $\pi_1(\bG)$ be the fundamental group of $\bG$.

Let $B\supseteq T$ be a Borel subgroup of $\bG_{\breve F}$ over $\breve F$ and contains $T$. Let $V=X_*(T)_{\G_0} \otimes \BR = X_*(S)\otimes\BR$ and let $V_+$ be the dominant chamber in $V$ determined by $B$. Let $\BS$ be the set of simple reflections in $W$ corresponding to $V_+$.

Let $\CA$ be the apartment of $\bG_{\breve F}$ corresponding to $S$. We fix a $\s$-stable alcove $\mathfrak{a}$ of $\CA$. By choosing a special vertex of $\mathfrak{a}$, we identify $\CA$ with $V$ such that $\mathfrak{a}$ lies in the dominant chamber $V_+$. Note that $\s$ does not necessarily preserve the special vertex. The action of $\s$ on $\CA$ induces an affine transformation on $V$. View $\s$ as an element in $\text{Aff}(V)=V \rtimes \GL(V)$ and let $p(\s)$ be the image of $\s$ under projection $\text{Aff}(V) \to \GL(V)$. Following \cite[\S2.1]{GHN2}, the $L$-action of $\s$ on $V$ is defined as $\s_0=w \circ p(\s)$, where $w$ is the unique element in $W$ such that $\s_0(V_+)=V_+$. Then $\s_0$ acts on $W$ and stabilizes $\BS$. Let $\BS/\<\s_0\>$ be the set of $\s_0$-orbits on $\BS$. For any $\l\in X_*(T)$, denote by $\underline \l $ the image of $\l$ in $X_*(T)_{\G_0}$ and by $\l^{\diamond}$ the average of the $\s_0$-orbit of $\underline\l$.

Let $\Phi$ be the relative root system of $\bG$ over $\breve F$. Let $\Sigma$ be the reduced root system associated with $\Phi$ as in \cite[\S1.7]{Tits1979}. We have a natural pairing $\BR\Sigma\times V \rightarrow \BR$. For any $i\in\BS$, let $\a_i\in \Sigma$, $\a_i^{\vee}\in V$ and $\omega_i \in \BR\Sigma$ be the corresponding simple root, simple coroot and fundamental weight of $\Sigma$ respectively. For any $\s_0$-orbit $o$ on $\BS$, set $\o_o=\sum_{i \in o} \omega_i$. Let $\rho$ be the half-sum of all positive roots in $\Sigma$. We have $\rho = \sum_{o\in\BS/\<\s_0\>}\o_o$.

Throughout this paper, all parabolic subgroups and Levi subgroups of $\bG$ are defined over $\breve F$. Let $J$ be a subset of $ \BS$. Denote by $P_J$ (resp. $M_J$) the corresponding standard parabolic (resp. Levi) subgroup of $\bG$. Denote by $W_J$ the subgroup of $W$ generated by element in $J$. For any Levi subgroup $M$ which contains $T$, set $W_{M} =\left( N_T(\breve F)\cap M(\breve F) \right)/T(\breve F)$, the Weyl group of $M$.

\subsection{The $\s$-conjugacy classes}\label{sec:1.2} The $\s$-conjugation action on $\breve G$ is defined by $g \cdot_\s g'=g g' \s(g) \i$. For $b \in \breve G$, we denote by $[b]$ the $\s$-conjugacy class of $b$. Let $B(\bG)$ be the set of $\s$-conjugacy classes of $\breve G$. 
The \emph{defect} of a $\s$-conjugacy class $[b]\in B(\bG)$ is defined as 
$$\de(b) = \mathrm{rank}_F(\bG)-\mathrm{rank}_F(\bJ_b),$$
where $\bJ_b$ is the $\s$-centralizer of $b$ and $\bJ_b(F)=\{g \in \breve G\mid g^{-1} b \s(g) =b\}$.

In \cite{Kot85} and \cite{Kot97}, Kottwitz gives a classification of the $\s$-conjugacy classes of $\breve G$. Let $\k:B(\bG)\rightarrow \pi_1(\bG)_{\G}$ be the Kottwitz map. Any $\s$-conjugacy class $[b]$ is determined by the two invariants: the element $\k(b)$ and the Newton point $\nu_b \in V_+^{\s_0}$.

The set $B(\bG)$ has a natural partial order: $[b_1]\le [b_2]$ if and only if $\k(b_1) = \k(b_2) $ and $\nu_{b_2} - \nu_{b_1}$ is the sum of positive coroots with non-negative coefficients. By \cite[Theorem 7.4 (iii)]{Chai}, the poset $B(\bG)$ is ranked, i.e., for any $[b_1], [b_2]\in B(\bG)$ such that $[b_1]\le[b_2]$, any maximal chains from $[b_1]$ to $[b_2]$ in $B(\bG)$ have the same length. We denote the common length by $\mathrm{length}([b_1],[b_2])$.

Let $\bG_{\ad}$ be the adjoint group of $\bG$. We have the natural map $B(\bG)\rightarrow B(\bG_{\ad})$. By \cite[Proposition 4.10]{Kot97}, the following commutative diagram is Cartesian

\begin{align*}
\xymatrix{B(\bG)\ar[d]^{\k}  \ar[r] &B(\bG_{\ad})\ar[d]^{\k_{\bG_{\ad}}}\\
\pi_1(\bG)_{\G}  \ar[r]&\pi_1(\bG_{\ad})_{\G},    }
\end{align*}
here $\k_{\bG_{\ad}}$ is the Kottwitz map of $\bG_{\ad}$.

We say that a $\s$-conjugacy class $[b] $ is basic if the Newton point $\nu_b$ is central. We say that a $\s$-conjugacy class $[b] $ is superbasic if $\de(b) = \sharp(\BS/\<\s_0\>)$.

In the rest of this subsection, we assume that $\bG$ is adjoint. Let $\bH$ be the quasi-split inner form of $\bG$ and let $\s_{\bH}$ be the Frobenius morphism on $\bH(\breve F)$. According to \cite[\S 2.3]{GHN1}, we can identify $\bG(\breve F)$ with $\bH(\breve F)$ such that $\s = \Ad(\g)\circ \s_{\bH}$ for some $\g\in N_T(\breve F)\subseteq \bG(\breve F)$. Here $\Ad(\g)$ is the map $g\mapsto \g g\g^{-1}$. Since $\Ad(\g)$ preserves the maximal torus $T$, we can identify $T$ with a maximal torus $T_{\bH}$ of $\bH$ over $F$. Note that $\pi_1(\bG)$ and $\pi_1(\bH)$ are equal as the $\G$-module. The induced action of $\s_\bH$ on $V$ is the same as $\s_0$.

For any $b'\in \bH(\breve F)$, denote by $[b']_{\bH}$ the $\s_\bH$-conjugacy class of $b'$. There is a natural bijection 
\begin{align*}\tag{1.1}\label{eq:1.1}
B(\bG)  \xrightarrow{\sim} B(\bH)    
\end{align*}
sending $[b]$ to $[b\g]_{\bH}$. Let $\k_{\bH}$ and $\nu^{\bH}$ be the Kottwitz map and Newton map of $\bH$ respectively. Let $[b]\in B(\bG)$. By \cite[\S 4.18]{Kot97}, we have $\nu_b =\nu^{\bH}_{b\g}$. By \cite[\S 4.9]{Kot97}, the following diagram is commutative
\begin{align*}
\xymatrix{
 B(\bG)\ar[d]^{\k}  \ar[r]^{(\ref{eq:1.1})} &B(\bH)\ar[d]^{\k_{\bH}}\\
\pi_1(\bG)_{\G}  \ar[r]^{\cdot \k(\g)}                 &\pi_1(\bH)_{\G}.               }
\end{align*}

For any $\l\in X_*(T ) $, we denote by $\k(\l) $ the natural image of $\l$ in $\pi_1(\bG )_{\G}$. Let $[b_1], [b_2]\in B(\bG)$ such that $[b_1]\le[b_2]$. Let $\l$ be an arbitrary dominant cocharacter in $X_*(T)$ such that $\k(\l) = \k(b)\cdot \k(\g)$ and $\nu_{b_2}\le \l^{\diamond}$ (such $\l$ always exists). By \cite[Theorem 7.4 (iv)]{Chai}, \cite[Proposition 3.11]{Ham15} and \cite[Theorem 3.4]{Viehmann2015} for quasi-split groups and the bijection (\ref{eq:1.1}), we have 
\begin{align*}\tag{1.2}\label{eq:1.2}
\leng([b_1],[b_2]) &= \sum_{o\in \BS/\<\s_0\>} \lceil \< \underline\l -\nu_{b_1} ,\o_o\>\rceil  - \lceil \< \underline\l  -\nu_{b_2} ,\o_o\>\rceil \\
&= \<\nu_{b_2}-\nu_{b_1},\rho\>+\frac{1}{2}\de(b_1)-\frac{1}{2}\de(b_2),
\end{align*}
where the notation $\lceil a\rceil$ means the minimal integer $\ge a$.

\subsection{The best integral approximation}\label{sec:1.3}

In this subsection, we assume that $\bG$ is adjoint and quasi-split over $F$. Then $\s=\s_0$.

Hamacher and Viehmann introduce the notion \textit{``best integral approximation"} of $\s$-conjugacy classes for unramified groups \cite[Lemma/Definition 2.1]{HV}. The definition also works for quasi-split groups. For any $[b] \in B(\bG)$, the best integral approximation $\l([b])$ is defined as the unique maximal element in $ X_*(T )_{\G } $ such that $ \k(\l ([b ]  ) ) = \k ( b  ) \in  \pi_1(\bG)_{\G }  $ and $\l ( [b ] ) \le \nu_{b } $ via the identification $ (X_*(T )_{\G_0})_{\s}  \otimes \BQ \cong ( X_*(T )_{\G_0}  \otimes \BQ ) ^{\s}\subseteq V$. 
%For any $[b]\in B(\bG)$, We set $\l([b]) = \l_{\bH}([b\g]_{\bH} )$ and call it the ``best integral approximation" of $[b]$. 

By \cite[Proposition 3.9 (iv)]{Schremmer2022_newton}, for any $[b]\in B(\bG)$, we have 
\begin{align*}\tag{1.3}\label{eq:1.3}
\de(b) = \sharp\{o\in \BS/\<\s_0\> \mid \<\nu_b,\o_o\> \ne \<\l([b]),\o_o\> \}.  
\end{align*}

Let $[b]\in B(\bG)$ and $o\in \BS/\<\s_0\>$. Let $\l\in X_*(T )$ be an arbitrary cocharacter such that $\k(\l) = \k(b)$. By the proof of \cite[Lemma/Definition 2.1]{HV}, we have, 
\begin{align*}\tag{1.4}\label{eq:1.4}
\<\nu_b,\o_o\> = \<\l([b]),\o_o\> \iff  \< \underline\l - \nu_b, \o_o\>\in\BZ. 
\end{align*}

For non quasi-split adjoint groups, we can define the best approximations of $\s$-conjugacy classes via the best integral approximations for the quasi-split inner forms and the natural isomorphism (\ref{eq:1.1}). In particular, (\ref{eq:1.3}), (\ref{eq:1.4}) also hold for non-quasi-split adjoint groups.
\subsection{The acceptable elements}\label{sec:1.4}

Let $\{\mu\}$ be a conjugacy class of cocharacters of $\bG$. Let $\mu\in X_*(T)^+$ be a
dominant representative of the conjugacy class $\{\mu\}$. Recall that for any $\l\in X_*(T)$, we denote by $\underline \l $ the image of $\l$ in $X_*(T)_{\G_0}$ and denote by $\l^{\diamond}$ the average of the $\s_0$-orbit of $\underline\l$.

The set of \emph{neutrally acceptable} elements for $\mu$ is defined as 
$$B(\bG,\{\mu\}) = \{[b]\in B(\bG) \mid \k(b) = \k(\mu), \nu_b \le \mu^{\diamond}\}.$$

It is easy to see that the set $B(\bG,\{\mu\})$ has the unique minimal element, which is the unique basic element $[b_{\mu, \mathrm{basic}}]$ in $B(\bG,\{\mu\})$. 

By definition, a $\s$-conjugacy class $[b]\in B(\bG,\{\mu\})$ is called \emph{$\mu$-ordinary} if $\nu_b = \mu^{\diamond}$.

It is shown in \cite[Theorem 2.1]{HN2018} that the set $B(\bG,\{\mu\})$ contains a unique maximal element $[b_{\mu,\max}]$. If $\s=\s_0$, then $[b_{\mu,\max}]$ is $\mu$-ordinary. If $\s\ne\s_0$, then this may not be the case.

For any $v\in V$, let $I(v) = \{i \in \BS\mid \<v,\a_i\>=0\} $. Let $[b]\in B(\bG,\{\mu\})$ and $J$ be a $\s_0$-stable subset of $\BS$. We say that $(\mu, [b])$ is \textit {Hodge-Newton decomposable} with respect to the standard Levi subgroup $M_J$ if $I(\nu_b) \subseteq J$ and $\mu^{\diamond} - \nu_{b} \in \sum_{j \in J} \BR_{\ge0} \a_j^\vee$.  We say that $(\mu,[b])$ is \textit {Hodge-Newton indecomposable}, if $(\mu, [b])$ is not Hodge-Newton decomposable with respect to $M_J$ for any $\s_0$-stable proper subset $J$ of $\BS$. The basic element $[b_{\mu,\mathrm{basic}}]$ is Hodge-Newton indecomposable. By \cite[\S 4.1]{HNY}, there is a unique maximal Hodge-Newton indecomposable element $[b_{\mu,\indec}]$.

We denote by $\nu_{\mu,\mathrm{basic}}$, $\nu_{\mu,\max}$ and $\nu_{\mu,\indec}$ the Newton point of $[b_{\mu,\mathrm{basic}}]$, $[b_{\mu,\max}]$ and $[b_{\mu,\indec}]$ respectively. 

Let $T_{\ad}$ be the image of $T$ in $\bG_{\ad}$ and $\mu_{\ad}$ be the image of $\mu$ in $X_*(T_{\ad})$. By \cite[\S6.5]{Kot97}, the natural map $B(\bG)\rightarrow B(\bG_{\ad})$ induces a bijection $B(\bG,\{\mu\})\cong B(\bG_{\ad},\{\mu_{\ad}\})$. For any $[b]\in B(\bG)$, denote by $[b_{\ad}]$ its image in $B(\bG_{\ad})$. If $[b]$ is $\mu$-ordinary, then $[b_{\ad}]$ is $\mu_{\ad}$-ordinary. If $(\mu,[b])$ is Hodge-Newton decomposable, then so is $(\mu_{\ad},[b_{\ad}])$.

\section{The essential gap function}

\subsection{Definition}\label{sec:2.1}
The goal of this section is to introduce the essential gap function on $B(\bG)$. Let $[b_1],[b_2]\in B(\bG)$ such that $ [b_1]\le [b_2]$. Define
   \begin{align*}
        \text{Ess-gap}([b_1],[b_2]) &= \<\nu_{b_2}-\nu_{b_1},2\rho\> - \mathrm{length}([b_1],[b_2]).
    \end{align*} 
By (\ref{eq:1.2}), we have
\begin{align*}
    \text{Ess-gap}([b_1],[b_2]) &= \mathrm{length}([b_1],[b_2])-\de(b_1)+\de(b_2) \\
                                &= \<\nu_{b_2} - \nu_{b_1},\rho\>-\frac{1}{2}\de(b_1)+\frac{1}{2}\de(b_2).
\end{align*}
It is easy to see that the essential gap function satisfies the additive property
$$\text{Ess-gap}([b_1],[b_2])+\text{Ess-gap}([b_2],[b_3])=\text{Ess-gap}([b_1],[b_3])$$
for $[b_1],[b_2],[b_3]\in B(\bG)$ with $[b_1]\le [b_2]\le [b_3]$.

\begin{lemma}\label{lem:2.1}
We have $\text{Ess-gap}([b_1],[b_2]) \in \BZ_{\ge0}$.
\end{lemma}
\begin{proof}
By the above additive property, it suffices to consider the case $\mathrm{length}([b_1],[b_2])=1$. By (\ref{eq:1.2}), $\<\nu_{b_2}-\nu_{b_1},2\rho\>\in\BZ$. Since $[b_2]>[b_1]$, we have $\<\nu_{b_2}-\nu_{b_1},2\rho\>>0$. Hence $\<\nu_{b_2}-\nu_{b_1},2\rho\>\ge1$. Then $\text{Ess-gap}([b_1],[b_2])\in\BZ_{\ge0} $ as desired.
\end{proof}

It is easy to see that the essential gap function is preserved under the natural map $B(\bG)\rightarrow B(\bG_{\ad})$.
\subsection{The $\GL_n$ case.}\label{sec:2.2}
In this subsection, we give a geometric interpretation of the essential gap function for the group $\GL_n$.  

In this case, we identify the coweight lattice $X_*(T)$ with $\BZ^n$. Let $[b]$ be a $\s$-conjugacy class. Write $\nu_b = (a_1,a_2,\ldots,a_n)$, $a_1\ge a_2\ge\cdots\ge a_n$. Following \cite[\S 1]{Chai}, the Newton polygon of $\nu_b$ is defined as the graph of the piecewise-linear function on the interval $[0,n]$ with slopes $(a_1,a_2,\ldots,a_n)$ and starts with the origin $(0,0)$. We identify $\nu_b$ with its Newton polygon.

By definition, a lattice point is a point in $\mathbb{R}^2$ with integral coordinates. Let $[b_1]<[b_2] $ be two $\s$-conjugacy classes with Newton points $\nu_{  1}, \nu_{ 2}$. Then $\nu_{ 1}$ lies below $\nu_{ 2}$. Let $P_{12}$ be the region between $\nu_1$ and $\nu_2$. Pick's theorem \cite{Pick} states that for the polygon $P_{12}$ (whose vertices are lattices), we have
    $$2\mathbf A = 2\mathbf i + \mathbf {b_1} + \mathbf {b_2},$$
where
    $$\mathbf A = \text{Area of }P_{12},$$
    $$\mathbf i = \text{number of lattice points in the interior of }P_{12},$$
    $$\mathbf{b_1} = \text{number of lattice points on the Newton polygon of } \nu_1 \text{ but not on the Newton polygon of }\nu_2,$$
    $$\mathbf{b_2} = \text{number of lattice points on the Newton polygon of } \nu_2 \text{ but not on the Newton polygon of }\nu_1.$$
%Here, we note that the common lattice points of $\nu_{b_1}$ and $\nu_{b_2}$ are not counted in $\mathbf {b_1}$ and $\mathbf {b_2}$.
Using (\ref{eq:1.2}), one can easily see that $\mathrm{length}([b_1],[b_2]) = \mathbf i + \mathbf {b_2}$. Note also that $\mathbf{A} = \<\nu_{ 2}-\nu_{ 1},\rho\>$. Then Pick's theorem implies 
$$\text{Ess-gap}([b_1],[b_2]) = \mathbf i + \mathbf {b_1}.$$

\begin{example}\label{eg1}
Let $\nu_{b_1} = (\frac{5}{4},\frac{5}{4},\frac{5}{4},\frac{5}{4},\frac{1}{4},\frac{1}{4},\frac{1}{4},\frac{1}{4})$ and $\nu_{b_2} = (3,1,\frac{1}{2},\frac{1}{2},\frac{1}{2},\frac{1}{2},0,0)$. Then $\mathbf {b_1} = 0$, $\mathbf {b_2} = 4$, $\mathbf i = 3$, $\mathbf A = 5$. Note that the common lattice points $(0,0), (4,5)$ and $(8,6)$ are not counted in $\mathbf {b_1}$ and $\mathbf {b_2}$. See Figure 1.    
\end{example}

%\newpage
\begin{figure}
\centering

\begin{tikzpicture}
\tkzInit[xmax=8,ymax=7,xmin=0,ymin=0]
\tkzGrid
\draw[ultra thick] (0,0) -- (4,5) -- (8,6);
\draw[ultra thick] (0,0) -- (1,3) -- (2,4) -- (6,6) -- (8,6);
%\draw[ultra thick] (0,0) -- (1,0) -- (8,3);

%\node at (1, 1) {$\nu_{b_2}$};
%\node at (3,1) ;

\end{tikzpicture}
\caption{}
\end{figure}

\subsection{The general case}\label{sec:2.3}
In this subsection, we assume that $\bG$ is adjoint. Let $\bH$ be the quasi-split inner form of $\bG$. As in \S\ref{sec:1.2}, we can identify $\bG(\breve F)$ with $\bH(\breve F)$ such that $\s=\Ad(\g)\circ \s_{\bH}$ on $\bG(\breve F)$ for some $\g\in N(\breve F)$. Recall that $V = X_*(T)_{\G_0}\otimes\BR$. For any $v\in V$, define $J(v) = \{o\in \BS/\<\s_0\> \mid \<v,\o_o\> =0\}$.

%$$I([b_1],[b_2])=\{o\in\BS/\<\s_0\>\mid \<\nu_{b_2}-\nu_{b_1},\o_o\>=0\},$$
%$$J([b_1],[b_2])=\{o\in\BS/\<\s_0\>\mid \<\nu_{b_2}-\nu_{b_1},\o_o\>\ne0\},$$
%$$J_1([b_1],[b_2]) = \{o\in J([b_1],[b_2])\mid \<\nu_{b_1} - \l (b_1 ), \o_o\> = 0  \},$$ 
%$$J_2([b_1],[b_2]) = \{o\in J([b_1],[b_2])\mid \<\nu_{b_2} - \l (b_2 ), \o_o\> = 0  \}.$$ 

\begin{definition}
Let $[b_1],[b_2]\in B(\bG)$ with $[b_1]\le [b_2]$. Write $\nu_1=\nu_{b_1}$ and $\nu_2=\nu_{b_2}$. Define
\begin{gather*}
\mathbf {b_1}([b_1],[b_2]) = \sharp \bigl(\{o\in\BS/\<\s_0\> \mid \< \nu_{1},\o_o\> = \< \l([b_1]) ,\o_o\> \}-J(\nu_{ 2}-\nu_{ 1})\bigr), \\
\mathbf {b_2}([b_1],[b_2]) = \sharp \bigl(\{o\in\BS/\<\s_0\> \mid \< \nu_{2},\o_o\> = \< \l([b_2]) ,\o_o\> \}-J(\nu_{ 2}-\nu_{1})\bigr), \\
\mathbf i([b_1],[b_2]) = \mathrm{length}([b_1],[b_2]) - \mathbf {b_2}([b_1],[b_2]).
\end{gather*}    
\end{definition}
By (\ref{eq:1.2}) and (\ref{eq:1.4}), one sees that the definitions of $\mathbf{b_1}$, $\mathbf{b_2}$ and $\mathbf{i}$ coincide with the definitions in $\S \ref{sec:2.2}$ for $\GL_n$. 
%The set $J(\nu_{ 2}-\nu_{ 1})=\{o\in\BS/\<\s_0\>\mid\<\nu_{ 1},\o_o\>=\<\nu_{ 2},\o_o\>\}$ is ``the overlap part" of $\nu_{ 1}$ and $\nu_{ 2}$, which are excluded in $\mathbf {b_1}([b_1],[b_2])$ and $\mathbf {b_2}([b_1],[b_2])$.

By (\ref{eq:1.3}), we have $\de(b_2)-\de(b_1) = -\mathbf{b_2}([b_1],[b_2]) + \mathbf{b_1}([b_1],[b_2]).$ Hence
\begin{align*}\tag{2.1}\label{eq:2.1}
\text{Ess-gap}([b_1],[b_2])  &= \leng([b_1],[b_2]) -\de(b_1)+\de(b_2)\\
                               &=  \leng([b_1],[b_2]) -\mathbf{b_2}([b_1],[b_2]) + \mathbf{b_1}([b_1],[b_2])\\
                               &= \mathbf {i}([b_1],[b_2]) + \mathbf{b_1}([b_1],[b_2]).
\end{align*}

%The following inequality is trivial in the $\GL_n$ case but it is nontrivial in the general case.

\begin{proposition}\label{prop:ige0}
Let $[b_1],[b_2]\in B(\bG)$ with $[b_1]\le[b_2]$. We have
$$\mathbf {i}([b_1],[b_2]) \ge 0.$$
In particular, $\text{Ess-gap}([b_1],[b_2]) = 0 $ if and only $\mathbf {b_1}([b_1],[b_2])=0$ and $\mathbf i([b_1],[b_2])=0$.

\end{proposition}
\begin{proof}
Write $\nu_1 = \nu_{b_1}$ and $\nu_2 = \nu_{b_2}$. Let $\l\in X_*(T )$ be an arbitrary dominant coweight such that $\k(\l) = \k(b)\cdot\k(\g)$ and $  \l^{\diamond}\ge\nu_{2}$. By (\ref{eq:1.4}), we have 
$$\mathbf {b_2}([b_1],[b_2]) = \sharp \bigl(\{o\in\BS/\<\s_0\> \mid \< \underline\l - \nu_2, \o_o\>\in\BZ \}-J(\nu_{ 2}-\nu_{1})\bigr).$$
By (\ref{eq:1.2}), we have
\begin{align*}
  \mathrm{length}([b_1],[b_2]) &= \sum_{o \in \BS/\<\s_0\> }(\lceil \<\underline\l  - \nu_{1} ,\o_o\>\rceil - \lceil\<\underline\l   - \nu_{2} ,\o_o\> \rceil).
\end{align*}
For $o\in J(\nu_{2}-\nu_{1})$, the summand vanishes. For $o\in  \{o\in\BS/\<\s_0\> \mid \<\underline\l -\nu_{ 2},\o_o\>\in\BZ\>\}-J(\nu_{ 2}-\nu_{ 1}) $, the summand is $\ge 1$. For $o\notin \{o\in\BS/\<\s_0\> \mid \<\underline\l -\nu_{2},\o_o\>\in\BZ\>\}$, the summand is $\ge 0$. Hence we have 
\begin{align*}
  \mathbf i([b_1],[b_2]) &= \mathrm{length}([b_1],[b_2]) - \mathbf{b_2}([b_1],[b_2])\ge0.
\end{align*}

The ``in particular" part follows from (\ref{eq:2.1}).
\end{proof}

Motivated by the $\GL_n$ case, we define the lattice points for $\s$-conjugacy classes as follows. Let $[b]\in B(\bG)$. We call $o\in \BS/\<\s_0\>$ a ``\textit{lattice point}'' of $[b]$ if $\<\nu_b,\o_o\> = \<\l([b]),\o_o\>$, or equivalently, $ \<\underline\l  -\nu_b,\o_o\>\in \BZ$ for some (or any) coweight $\l\in X_*(T)$ with $\k(\l)  = \k(b)\cdot\k(\g)$ (here we use (\ref{eq:1.4})). Roughly speaking, $\mathbf {i}([b_1],[b_2])$ is the number of ``interior lattice points'' between $\nu_{1}$ and $\nu_{2}$, $\mathbf {b_1}([b_1],[b_2])$ and $\mathbf {b_2}([b_1],[b_2])$ are the numbers of ``upper and lower lattice points'' respectively, while the ``common lattice points'', that is, those lattice points lying in $J(\nu_{ 2}-\nu_{ 1})$ are not counted.

\section{Affine Deligne--Lusztig varieties}\label{sec:adlv}
\subsection{(Single) affine Deligne--Lusztig variety} \label{sec:3.1}
Let $\breve \CI$ be the $\s$-stable Iwahori subgroup corresponding to the base alcove $\mathfrak{a}$. Let $\tW= N_T(\breve F)/(T(\breve F) \cap \breve \CI)$ be the Iwahori--Weyl group. The choice of special vertex of $\mathfrak{a}$ (see \ref{sec:1.1}) gives rise to a bijection
$$\tW \cong X_*(T)_{\G_0} \rtimes W=\{t^{\underline\l} w\mid \underline\l \in X_*(T)_{\G_0}, w \in W\}.$$
For any $w\in \tW$, we fix a representative $\dot{w}\in N_T(\breve F)$. Let $\tilde \BS$ be the set of simple reflections in $\tW$ determined by $\mathfrak{a}$ and let $\ell$ be the corresponding length function. Let $\Omega$ be the set of length zero elements in $\tW$. Then $\tW = W_a\rtimes \Omega$, where $W_a $ is the affine Weyl group. We can identify $\Omega$ with $\pi_1(\bG)_{\G_0}$. For any Levi subgroup $M$, set $\tW_M = X_*(T)_{\G_0} \rtimes W_M$.

For any $b \in \breve G$ and $w \in \tW$, we define the corresponding \textit {affine Deligne--Lusztig variety} corresponding to $w$ and $b$ as
$$X_w(b)=\{g \breve \CI \in \breve G/\breve \CI\mid g \i b \s(g) \in \breve \CI \dot w \breve \CI\} .$$
Note that $X_w(b)$ is a subscheme, locally of finite type, inside the affine flag variety $\breve G/\breve \CI$.
%Let $\kk$ be the residue field of $\breve F$. In the equal characteristic setting, the affine Deligne--Lusztig variety $X_w(b)$ is the set of $\kk$-valued points of a locally closed subscheme of the affine flag variety, equipped with the reduced scheme structure. In the mixed characteristic setting, we consider $X_w(b)$ as the $\kk$-valued points of a perfect scheme in the sense of Zhu \cite{Zhu} and Bhatt--Scholze \cite{BS}, a locally closed perfect subscheme of the $p$-adic partial flag variety. 

%We denote by $\Sigma^{\mathrm{top}}(X_w(b))$ the set of top-dimensional irreducible components of $X_w(b)$. 
If $b$ and $b'$ are $\s$-conjugate in $\breve G$, then $X_w(b)$ and $X_w(b')$ are isomorphic. Thus, the affine Deligne--Lusztig variety $X_w(b)$ (up to isomorphism)  depends only on the element $w\in\tW$ and the $\s$-conjugacy class $[b]$. 

Let $B(\bG)_w = \{ [b]\in B(\bG) \mid X_w(b)\ne \emptyset\}$. There is a unique maximal element $[b_{w,\max}]$ in $B(\bG)_w$. By~\cite[Lemma 3.2]{MV}, $\dim X_{w}(b_{w,\max}) = \ell(w) - \<{\nu}_{b_{w,\max}},2\rho\>$.

Let $\tW_{\ad}$ be the Iwahori--Weyl group corresponding to $\bG_{\ad}$ and $T_{\ad}$. For any $\t\in \Omega$, let $X_w(b)^{\t}$ be the intersection of $X_w(b)$ with the connected components of the affine flag variety corresponding to $\t$. Then the natural map $\bG\rightarrow\bG_{\ad}$ induces an isomorphism $X_w(b)^{\t}\cong X_{w_\ad}(b_{\ad})^{\t_{\ad}}$, where $w_{\ad}$ (resp. $\t_{\ad}$) is the image of $w$ (resp. $\t$) under the natural map $\tW\rightarrow\tW_{\ad}$.

\begin{proposition}\label{dimXwb}
Let $[b]\in B(\bG)_w$, we have 
$$\dim X_w(b)\ge \dim X_w(b_{w,\max}) +\text{Ess-gap}([b],[b_{w,\max}])  . $$

\end{proposition}

\begin{proof}
We first consider the equal characteristic case. Let 
$$X = \breve\CI w \breve\CI. $$

Following \cite[\S2.5]{He-CDM}, we introduce the notion of bounded admissible subset of $\breve G$. Let $\breve \CI_r$ be the $r$-th Moy-Prasad subgroup associated with the barycenter of the base alcove. Denote by $\pi_r: \breve G \to \breve G / \breve\CI_r$ the natural quotient map. A subset $Y$ of $\breve G$ is called bounded \textit{admissible} if there exist $r \in \BN$, $w \in \tW$ and a locally closed subvariety $Z$ of $\cup_{w' \leq w} \breve\CI w' \breve\CI / \breve\CI_r$ such that $Y = \pi_r\i(Z)$. In this case, the closure of $Y$ is defined as $\pi_r\i(\overline Z)$, where $\overline Z$ is the closure of $Z$ in $\cup_{w' \leq w} \breve\CI w' \breve\CI / \breve\CI_r$. Furthermore, we say that a subset $Y_1$ of $Y$ is an irreducible component of $Y$ if $Y_1 = \pi_r\i(Z_1)$ with $Z_1$ an irreducible component of $Z$. We define the relative dimension
$$\dim_{\breve \CI} Y_1 = \dim(Y_1/\breve \CI_r ) - \dim (\breve \CI /\breve \CI_r)$$ and
$$\dim_{\breve \CI} Y = \max_{Y_1}\dim_{\breve \CI} Y_1,$$
where $Y_1$ runs over all irreducible components of $Y$.

Note that $X$ and $X\cap [b]$ are admissible subsets and $\dim_{\breve \CI}X\cap [b] = \dim X_w(b) + \<\nu_b,2\rho  \>$. By \cite[Propositon 1 and Corollary 1]{Ham17}, the Newton stratification of $X$ satisfies topological strong purity in the sense of \cite[Definition 1]{Ham17}. Let $C_0$ be an irreducible component of $X\cap [b]$ of maximal (relative) dimension. Apply the argument as in \cite[Proposition 5.13]{Ham15} or \cite[Lemma 5.12]{Viehmann2015}, there exist a chain $[b] = [b_0]<[b_1]<\cdots < [b_k] $ and irreducible components $C_i$ of $X\cap [b_i]$ for $i=1,2,\ldots,k$ such that

\begin{itemize}
\item The closure $\overline{C_k}$ is an irreducible component of $X$,
\item $C_i\subseteq \overline{C_{i+1}}$ and the codimension of $C_i$ in $\overline{C_{i+1}}$ is $1$ for $i=0,1,\ldots,k-1$.
\end{itemize}
Since $X$ is irreducible, $[b_k]=[b_{w,\max}]$. Then the codimension of $C_0$ in $\overline{X\cap [b_{w,\max}]} = X$ is $\le\mathrm{length}([b],[b_{w,\max}])$. Hence
\begin{align*}
    (\dim X_w(b_{w,\max}) + \<\nu_{b_{w,\max}},2\rho\>) - ( \dim X_w(b) + \<\nu_b,2\rho\>)&= \dim_{\breve \CI}X - \dim_{\breve \CI}C_0\\
    &\le \mathrm{length}([b],[b_{w,\max}]).
\end{align*}
Thus,
$$\dim X_w(b)\ge \dim X_w(b_{w,\max}) +\text{Ess-gap}([b],[b_{w,\max}])  . $$

Note that the statement can be translated into properties of extended affine Weyl group via class polynomial or Deligne--Lusztig reduction (see \cite[Theorem 6.1]{He14}). Hence it is also true in the mixed characteristic case. 
\end{proof}

%\begin{remark}
%For general $[b_1],[b_2]$ with $[b_1]<[b_2]$, it may not be true that    
%$$\dim X_w(b_1)\ge \dim X_w(b_{2}) +\text{Ess-gap}([b_1],[b_2])  . $$
%\end{remark}

\subsection{Certain union of affine Deligne--Lusztig varieties}\label{sec:3.2}
The $\{\mu\}$-admissible set is defined as
$$\Adm(\{\mu\}) = \{w\in\tW\mid w\le t^{x(\underline\mu)}\text{ for some }x\in W\}.$$ %\remind{address Michael's concern}
Let $K\subseteq \tilde{\BS}$ be a $\s$-stable subset such that $W_K$ is finite. Let $ \breve \CK $ be the standard parahoric subgroup of $G(\breve F)$ corresponding to $K$. We set
$$X(\mu,b)_K = \{  g\breve \CK \mid g^{-1}b\s(g) \in \breve \CK \Adm(\{\mu\})\breve \CK   \}.$$
We simply write $X(\mu,b)$ for $X(\mu,b)_{\emptyset}$. By \cite{He16}, $X(\mu,b)_K \ne \emptyset$ if and only if $[b] \in B(\bG,\{\mu\}) $.

\begin{proposition}\label{prop:dimxmub}
For any $[b]\in B(\bG,\{\mu\})$, we have 
$$\dim X(\mu,b)_K \ge \<\underline\mu -\nu_{\mu,\max}, 2\rho\> +\text{Ess-gap}([b],[b_{\mu,\max}]).$$

\end{proposition}
\begin{proof}
The proof is similar to the proof of Proposition \ref{dimXwb}. We only need to consider the equal characteristic case. Let
$$X = \bigsqcup_{w \in \Adm(\{\mu\})} \breve  \CK w \breve  \CK. $$ 
Then $X$ and $X\cap[b]$ are bounded admissible subsets of $\breve G$. For any bounded admissible subset $Y$ of $\bG$, we define $\dim_{\breve \CK}(Y) = \dim_{\breve \CI} (Y) - \dim_{\breve \CI} (\breve \CK)$. Note that for each irreducible component $C$ of $X$, one has $\dim _{\breve \CK}C = \<\underline\mu,2\rho\>$ (see \cite[\S2.5]{He-CDM} and \cite[Proposition~2.1]{HN17}). Moreover, we have $\dim_{\breve \CK}X\cap [b] = \dim X(\mu,b)+\<\nu_b,2\rho\>$.

Let $C_0$ be an irreducible component of $X\cap [b]$ such that $\dim_{\CK}C_0$ is maximal. Similar to the proof of Proposition \ref{dimXwb}, there exist some $[b']\in B(\bG,\{\mu\})$ and an irreducible component $C'$ of $X\cap [b']$ such that 
\begin{itemize}
\item The closure $\overline{C'}$ is an irreducible component of $X$;
\item $C_0\subseteq \overline{C'}$ and the codimension of $C_0$ in $\overline{C'}$ is $\le \mathrm{length}([b],[b']) \le \mathrm{length}([b],[b_{\mu,\max}])$.
\end{itemize}
We point out that $[b']$ is not necessarily equal to $[b_{\mu,\max}]$. Then
\begin{align*}
\dim X(\mu,b)_K + \<\nu_b,2\rho\> & = \dim_{\breve \CK} C_0 \\
                               &\ge  \<\underline\mu,2\rho\> -\mathrm{length}([b],[b_{\mu,\max}])\\
                               &= \<\underline\mu-\nu_{\mu,\max},2\rho\> + \<\nu_{\mu,\max},2\rho\> - \mathrm{length}([b],[b_{\mu,\max}])
\end{align*}
Therefore, 
\[\dim X(\mu,b)_K \ge \<\underline\mu -\nu_{\mu,\max}, 2\rho\> +\text{Ess-gap}([b],[b_{\mu,\max}]). \qedhere\]
\end{proof}

If $\mu$ is non-central in every quasi-simple component of $\bG$ over $F$, then Proposition \ref{prop:dimxmub} recovers the inequality in \cite[Theorem 3.5]{GHN3}.

%implies
%\begin{align*}
%\dim X(\mu,b_{\mu,\basic})_K &\ge \<\underline\mu -\nu_{\mu,\max}, 2\rho\> +\text{Ess-gap}([b_{\mu,\basic}],[b_{\mu,\max}])\\
%&=\<\underline\mu -\nu_{\mu,\max}, 2\rho\> +\mathbf{i}([b_{\mu,\basic}],[b_{\mu,\max}])+\mathbf{b_1}([b_{\mu,\basic}],[b_{\mu,\max}])\\
%&\ge \mathbf{b_1}([b_{\mu,\basic}],[b_{\mu,\max}])\\
%&= \sharp(\BS/\<\s_0\>) - \de(b_{\mu,\basic}).
%\end{align*}

\subsection{Hodge-Newton decomposition}\label{sec:HN} Let $J$ be a $\s_0$-stable subset of $\mathbb{S}$. Let $\mathfrak{P}_J^{\s}$ be the set of $\s$-stable parabolic subgroups which conjugate to the standard parabolic subgroup $P_J$ and contain $T$. For any $P\in \mathfrak{P}_J^{\s}$, there is a unique Levi factor $M$ of $P$ which contains $T$. We call $P = MN $ the semi-standard Levi decomposition of $P$, where $N$ is the unipotent radical of $P$. There is a unique element $z_P$ in $W^{J}$ with ${}^{z_{P}}P_J = P$ and ${}^{z_{P}}M_J = M$. Set $\mu_{P} = z_{P}(\mu)$ and let $\{\mu_{P}\}$ be the $W_M$-conjugacy class of $\mu_P$.

Let $[b]\in B(\bG,\{\mu\})$. Assume that $(\mu, [b])$ is Hodge-Newton decomposable with respect to $M_J$ (see \S\ref{sec:1.4} for the definition). By \cite[Proposition 4.8]{GHN2}, there exists a unique element $[b_P]_M\in B(M,\{\mu_P\})$, such that $[b_P]_{M}\subseteq [b]$ and the $M$ dominant Newton point $\nu_{b_P}^M$ of $b_P$ is equal to $z_{P}(\nu_b)$.

We have the following important Theorem.

\begin{theorem}[{\cite[Theorem 4.17]{GHN2}}]\label{thm:HN-decom}
Let $K$ be a $\s$-stable subset of $\tS$ such that $W_K$ is finite and let $J$ be a $\s_0$-stable subset of $\BS$. Let $[b]\in B(\bG,\{\mu\})$ such that $(\mu,[b])$ is Hodge-Newton decomposable with respect to $M_J$. Then 
$$X(\mu,b)_K \cong\bigsqcup_{P=MN  \in \mathfrak{P}_J^{\s}/W_K^{\s}}X^{M}(\mu_{P}, b_P)_{K_{M}},$$
where $P = MN$ is the semi-standard Levi decomposition of $P$ and $K_{M} $ be the set of simple reflections of $W_{M}$ generating $W_K\cap W_{M}$.  
\end{theorem}

Let $P,P' \in \mathfrak{P}_J^{\s}$ with the semi-standard Levi decompositions $P=MN,P'= M' N' $. By \cite[Lemma 4.9]{GHN2}, there exists $u\in \tW^{\s}$ such that ${}^{u}P = P'$, ${}^{u}M = M'$ and ${}^{u}  \breve {\CI}_{M}   =  \breve {\CI}_{M'}$, where, for any Levi subgroup $  L$ of $\bG$, we denote $\breve \CI_{  L} = \breve \CI\cap   L(\breve F)$. As $u$ is $\s$-stable, the following diagram is commutative
\begin{align*}\tag{3.1}\label{eq:3.1}
\begin{xymatrix}
{M(\breve F)\ar[d]^{\s}  \ar[r]^{\Ad(u)} & M'(\breve F)\ar[d]^{\s}\\
M(\breve F)  \ar[r]^{\Ad(u)}&  M'(\breve F).}
\end{xymatrix}
\end{align*}
Here $\Ad(u)$ is the map $g\mapsto ugu^{-1}$. Thus, we have a bijection
\begin{align*}\tag{3.2}\label{eq:3.2}
\Ad(u):B(M,\{\mu_P\})\xrightarrow{\sim} B(M',\{\mu_{P'}\}).
\end{align*}

In the rest of this section, we assume that $\bG$ is adjoint. Let $\bH$ be the quasi-split inner form of $\bG$. As in \S\ref{sec:1.2}, we can identify $\bG(\breve F)$ with $\bH(\breve F)$ such that $\s=\Ad(\g)\circ \s_{\bH}$ for some $\g\in N(\breve F)$. Then $\g$ is a lift of some length-zero element $\t\in \tW$. The induced action of $\s_\bH$ on $V$ is the same as $\s_0$ (see \S\ref{sec:1.2}). We can identify $\tW$ with the Iwahori--Weyl group of $\bH$. Then $\s$ acts on $\tW$ as $\Ad(\t)\circ\s_0$.

Let $P \in \mathfrak{P}_J^{\s}$ with the semi-standard Levi decomposition $P=MN$. We now study the Frobenius action $\s\vert_{\tW_M}$ on $\tW_M$. Since $^{z_P}P_J=P$ and $P$ is $\s$-stable, the element $z_{P}^{-1} \t \s_0(z_{P})$ lies in $\tW_J$ (see \cite[Lemma 4.7]{GHN2}).

%\remind{ I am afraid that the reader think of $\t\s_0(z_P)$ as $(\Ad(\t)\circ\s_0)z_P$. Would it be better to write $\t\cdot \s_0(z_P)$?? -- QY }
\begin{lemma}\label{lem:tauJ}

Let $P\in\mathfrak{P}_J^{\s}$ with the semi-standard Levi decomposition $P=MN$. Let $\t_J = z_{P}^{-1} \t \s_0(z_{P})$. Then
\begin{enumerate}
    \item the element $\t_J$ is length zero in $\tW_J$.
    \item the Frobenius morphism $\s\vert_{\tW_M }:\tW_M \rightarrow \tW_M $ equals $\Ad(z_{P})\circ\Ad(\t_J)\circ \s_0 \circ \Ad(z_{P}^{-1})$. In particular, the $L$-action $\s_0^{M}$ of $\s\vert_{M}$ on $V$ is equal to $\Ad(z_{P})\circ \s_0 \circ \Ad(z_{P}^{-1})$, and the $\s_0^{M}$-average $\mu_{P}^{\diamond,M}$ of $\mu_{P}$ is equal to $z_{P}(\mu^{\diamond})$.
\end{enumerate}
\end{lemma}

\begin{proof}
Write $\t = t^{\underline\l}y$ for $\underline \l\in X_*(T)_{\G_0}$ and $y\in W$. Write $z=z_{P}$ for simplicity. 
\begin{enumerate}
\item We have
\begin{align*}
\ell (\t ) & =\sum_{\substack{\b\in\Phi ^{+}\\  y^{-1} \b>0}}\lvert  \<\underline\l, \b\> \rvert +\sum_{\substack{\b\in\Phi ^{+}\\  y^{-1} \b<0}}\lvert  \<\underline\l, \b\>-1 \rvert,
\end{align*}
and
\begin{align*}
\ell_J(\t_J) &=\ell_J(  t^{z^{-1} (\underline\l)  }\cdot z^{-1}y\s_0(z) )\\
&=\sum_{\substack{\a\in\Phi_J^{+}\\ \s_0(z^{-1})y^{-1}z\a>0}}\lvert  \<\underline\l,z\a\> \rvert +\sum_{\substack{\a\in\Phi_J^{+}\\ \s_0(z^{-1})y^{-1}z\a<0}}\lvert  \<\underline\l,z\a\>-1 \rvert.
\end{align*}
Note that $y^{-1}z\a=\s_0(z) (\s_0(z^{-1})y^{-1}z)\a$. Since $(\s_0(z^{-1})y^{-1}z)\in W_J$ and $\s_0(z)\in W^{J}$, the roots $y^{-1}z\a$ and $\s_0(z^{-1})y^{-1}z\a$ have the same signs. Then $\ell_J(\t_J)\le\ell(\t)=0$.

\item Let $m\in \tW_M$. Direct computation shows that
\begin{align*}
  \Ad(z)\circ\Ad(\t_J)\circ \s_0 \circ \Ad(z^{-1})(m) &= \Ad(\t)\circ\Ad(\s_0(z))\circ \s_0 \circ \Ad(z^{-1})(m)\\
  &=\Ad(\t)\circ\s_0(m).  
\end{align*}
Using (1), we get $\s_0^{M} = \Ad(z_{P})\circ \s_0 \circ \Ad(z_{P}^{-1})$. Moreover, we have 
$$\left(\s_0^{M}\right)^{i}(\mu_{P}) = \left(\Ad(z)\circ \s_0 \circ \Ad(z^{-1})\right)^{i}(\mu_{P})=z\s_0^{i}(\mu).$$
Then the statement follows.  \qedhere
\end{enumerate} 
\end{proof}

The following proposition is a key step in the proof of Theorem \ref{thm:main}. Recall that a $\s$-conjugacy class $[b]\in B(\bG,\{\mu\})$ is called $\mu$-ordinary if $\nu_b=\mu^{\diamond}$ (see \S\ref{sec:1.3}).

\begin{proposition}\label{prop:levi}
Let $P\in\mathfrak{P}_J^{\s}$ with the semi-standard Levi decomposition $P=M N$. Let $[b_{\mu_{P},\max} ]_{M}$ be the maximal element in $B(M,\{\mu_{P}\})$. Suppose $[b_{\mu,\max}]$ is $\mu$-ordinary. Then 
\begin{enumerate}
    \item $[b_{\mu_{P},\max} ]_{M}\subseteq [b_{\mu,\max}]$ and $[b_{\mu_{P},\max} ]_{M}$ is $\mu_{P}$-ordinary, i.e., the $M$-dominant Newton point of $[b_{\mu_{P},\max} ]_{M}$ equals $z_{P}(\mu^{\diamond})$.
    \item Let $[b]\in B(\bG,\{\mu\})$ such that $(\mu,[b])$ is Hodge-Newton decomposable with respect to $M_J$. Then
    \begin{gather*}
      \mathrm{length}([b],[b_{\mu,\max}]) = \mathrm{length}_{M}([b_P]_M,[b_{\mu_{P},\max} ]_{M});\\
      \text{Ess-gap}([b],[b_{\mu,\max}]) = \text{Ess-gap}_{M}([b_P]_M,[b_{\mu_{P},\max} ]_{M}),
    \end{gather*}
    where $\mathrm{length}_{M}$ and $\text{Ess-gap}_{M}$ are the length function and the essential gap function on $B(M)$, and $[b_P]_M$ is the $M(\breve F)$-$\s$-conjugacy class of $b_P$.

\end{enumerate}
\end{proposition}

\begin{proof} 
 \begin{enumerate} 
\item Recall that $\t$ is the length-zero element in $\tW$ such that the action of $\s$ on $\tW$ equals $\Ad(\t)\circ \s_0$. Write $\t=t^{\underline\l}y$ where $\underline\l\in X_*(T)_{\G_0}$ and $y\in W$. Recall that for any $v\in V$, we denote $I(v) = \{i \in \BS\mid \<v,\a_i\>=0\} $. Write $z=z_{P}$ for simplicity. For any $i\in J$, let $\o_i^J\in \sum_{j\in J}\BR \a_j$ be the fundamental weight of $M_J$ corresponding to $i$. For any $o\in J/\<\s_0\>$, let $\o_o^J = \sum_{i\in o}\o_i^J$. By Lemma \ref{lem:tauJ} (2), the $L$-action $\s_0^{M}$ of $\s\vert_{M}$ is equal to $\Ad(z)\circ \s_0 \circ \Ad(z^{-1})$. Thus $z(\o_o^J),o\in \BS/\<\s_0\>$ are exactly the sums of the $\s_0^{M}$-orbits of fundamental weights of $M$. 

By \cite[Corollary 3.6]{HN2018}, $[b_{\mu,\max}]$ is $\mu$-ordinary if and only if $\< \underline\l ,\o_o\>\in \BZ$ for any $o \in \bigl(\BS - I(\mu^{\diamond} )\bigr)/\<\s_0\>$. Similarly, $[b_{\mu_P,\max}]_M$ is $\mu_P$-ordinary if and only if $\<z^{-1}( \underline\l ), \o_o^J \>\in \BZ$ for any $o \in \bigl(J - I(\mu^{\diamond} )\bigr)/\<\s_0\>$. Since the action of $\s$ on $\tW$ is of finite order, there exists some $N\in\BN$ such that $(\t\s_0)^N = 1$ and hence $(t^{z^{-1}(\underline\l) }z^{-1}y\s_0z)^N = (z^{-1}\t\s_0z)^N = 1$. Since $z^{-1}y \s_0 z \in W_J\rtimes \<\s_0\>$, we have $z^{-1}(\underline\l)\in \sum_{j\in J}\BR\a_j^{\vee}$. It follows that $\<z^{-1}( \underline\l) , \o_o^J \> = \<z^{-1}(\underline\l),\o_o\> $ for any $o\in J/\<\s_0\>$. Note that $z^{-1}(\underline\l) \in \underline\l + \sum_{i\in\BS}\BZ\a_i^{\vee}$. Thus, the condition that $[b_{\mu,\max}]$ is $\mu$-ordinary implies that $[b_{\mu_P,\max}]_M$ is $\mu_P$-ordinary.

%\item There is some $\mu'$ in the $W$-orbit of $\mu$ such that $t^{\mu'}\in [b_{\mu,\max}]$. Recall that $p(\s)$ is the image of $\s$ under the natural projection $\text{Aff}(V)\rightarrow \GL(V)$ (cf. \S\ref{sec:1.1}). Since $\nu_{\mu,\max}=\mu^{\diamond}$, we have $\frac{1}{N}\sum_{i=0}^{N-1} p(\s)^i (\mu') = y(\mu^{\diamond})$ for some $y\in W^{I(\mu^{\diamond})}$ and $N\in\BN$ (see for example \cite[\S1.3]{HN2018}). Then $p(\s)(y(\mu^{\diamond})) = y(\mu^{\diamond})$ and hence $p(\s)(y)=y$. Write $y = z_0w_J$ for some $z_0\in W^J$ and $w_J\in W_J$. Let $P_0 = {}^{z_0}P_J$ and $M_0 = {}^{z_0}M_J$. Then $P_0\in\mathfrak{P}_J^{\s}$. Note that $t^{\mu'}\in \tW_{M_0}$ whose $M_0$-dominant Newton point is $(z_0w_J^{-1}z_0^{-1})y(\mu^{\diamond}) = z_0(\mu^{\diamond})$. Combined with Lemma \ref{lem:tauJ} (2), this proves (1) for $P_0\in \mathfrak{P}_J^{\s}$. By (\ref{eq:3.2}), (1) is true for any given $P\in \mathfrak{P}_J^{\s}$.

\item By (\ref{eq:1.2}) and the assumption that $(\mu,[b])$ is Hodge-Newton decomposable with respect to $M_J$, we get 
\begin{align*}
\mathrm{length}([b],[b_{\mu,\max}]) &= \<\mu^{\diamond}-\nu_b,\rho\> +\frac{1}{2}\de(b)- \frac{1}{2}\de(b_{\mu,\max})  \\
&= \<\mu^{\diamond}-\nu_b,\rho_{M_J}\> +\frac{1}{2}\de(b)- \frac{1}{2}\de(b_{\mu,\max})  \\
&= \<z(\mu^{\diamond})-z(\nu_b),\rho_M\> +\frac{1}{2}\de_M(b_P)- \frac{1}{2}\de_M(b_{\mu_P,\max}).
% &= \sum_{o\in J/\<\s_0\>} \lceil\<z(\l)-z(\nu_b),z(\o_o)\>  \rceil -\lceil\<z(\l)-z(\mu^{\diamond}) ,z(\o_o)\> \rceil.
\end{align*}
By definition, the $M$-dominant Newton point of $b_P$ is equal to $z(\nu_b)$. Note that (\ref{eq:1.2}) also works for $M$. Combined with (1), we get $\mathrm{length}([b],[b_{\mu,\max}])=\mathrm{length}_{M}([b_P]_M,[b_{\mu_{P},\max} ]_{M})$, and the equality for essential gap also follows.
\qedhere
%Next, we have
%\begin{align*}
% \text{Ess-gap}([b],[b_{\mu,\max}]) &=  \<\mu^{\diamond} - \nu_{b },2\rho\>-\mathrm{length}([b],[b_{\mu,\max}])\\
%                                    &=  \<z(\mu^{\diamond}) - z(\nu_{b }),2\rho_{M}\> - \mathrm{length}_{M}([b_P]_M,[b_{\mu_{P},\max}]_{M})\\
%                                    &=\text{Ess-gap}_{M}([b_P]_M ,[b_{\mu_{P},\max}]_{M}),
%\end{align*}
%where the last equation follows from (1). \qedhere
\end{enumerate}\end{proof}

\section{Classification of zero-dimensional ADLV}
\subsection{Extended Lubin-Tate type} We first recall the notion of of Tits datum and Coxeter datum from \cite[Definition 2.6]{GHR}. A {\it Tits datum} over $F$ is a tuple $(\tilde{\Delta}, \d, \{\l\})$, where $\tilde{\Delta}$ is an absolute local Dynkin diagram, and $\d$ is a diagram automorphism of $\tilde{\Delta}$, and $\{\l\}$ is a $W$-conjugacy class in $X_*$, the coweight lattice of the reduced root system associated with $\tilde{\Delta}$. A {\it Coxeter datum} over $F$ is a tuple $((W_a,\BS_a),\d,\{\l\})$, where $(W_a,\BS_a)$ is an affine Coxeter system, $\d$ is a length-preserving automorphism of $W_a$ and $\{\l\}$ is a Weyl group orbit in the coweight lattice of the root system of $(W_a,\BS_a)$. A Tits datum gives rise to a Coxeter datum by forgetting the arrows in the local Dynkin diagram. There may be different Tits data corresponding to the same Coxeter datum.

We define the Tits datum associated with the pair $(\bG,\{\mu\})$ as follows. Let $\tilde{\Delta}_{\bG}$ be the local Dynkin diagram associated with $\bG_{\breve F}$ (cf. \cite{Tits1979}) and let $\d_{\bG}$ be the diagram automorphism on $\tilde{\Delta}_{\bG}$ induced by $\s$. Let $\underline{\mu_{\ad}}$ be the image of $\mu$ under the natural map $X_*(T)\rightarrow X_*(T)_{\G_0} \rightarrow X_*(T_{\ad})_{\G_0}$. Note that $X_*(T_{\ad})_{\G_0}$ is torsion free. By \cite[Lemma 15]{HR}, we have a natural injection $ X_*(T_{\ad})_{\G_0} \hookrightarrow X_*$. We also denote by $\underline{\mu_{\ad}}$ its image in $X_*$. We call $(\tilde{\Delta}_{\bG}, \d_{\bG}, \{\underline{\mu_{\ad}}\})$ the Tits datum associated with $(\bG,\{\mu\})$. The Coxeter datum associated with $(\bG,\{\mu\})$ is the one corresponding to $(\tilde{\Delta}_{\bG}, \d_{\bG}, \{\underline{\mu_{\ad}}\})$.

A Tits datum $(\tilde{\Delta}, \d,\{\l\})$ over $F$ is called {\it of extended Lubin-Tate type} if every quasi-simple component on which $\l$ is non-central is of the form $(\mathrm{Res}_{F_d/F}({A}_n, \id), \{(\omega_1^{\vee},0,\ldots,0)\})$, where $F_d/F$ is the unramified extension of degree $d$. Note that this definition is slightly different from that in \cite[\S2.6]{GHR}. In fact, if $\l$ is central, according to our definition, the associated Tits datum is of extended Lubin-Tate type.

It is well known that if the Tits datum associated with $(\bG,\{\mu\})$ is of extended Lubin-Tate type, then $\dim X(\mu,b)_K = 0$ for any $[b]\in B(\bG,\{\mu\})$ and any $\s$-stable $K\subseteq \tilde{\BS}$ with $W_K$ finite.

Following \cite[Definition 3.1]{GHN2}, we say that $(\bG, \{\mu\})$ is \textit{fully Hodge-Newton decomposable}, if every non-basic $\s$-conjugacy class $[b]\in B(\bG,\{\mu\})$ is Hodge-Newton decomposable with respect to some proper standard Levi subgroup.

We would like to point out the following observation. By definition, $[b_{\mu,\indec}]$ is the unique maximal Hodge-Newton indecomposable element. Thus $(\bG,\{\mu\})$ is fully Hodge-Newton decomposable if and only if $[b_{\mu,\mathrm{basic}}]=[b_{\mu,\indec}]$.
%Hence, we have isomorphism of triples
%$$ (M,  \s^{J}_{P'}, \{\mu\})\simeq (M', \s, \{\mu_{P'}\})\simeq(M'', \s, \{\mu_{P''}\})\simeq (M,   \s^{J}_{P''}, \{\mu\})$$

\begin{proposition}\label{prop:nu0superbasic}
Suppose $\bG$ is quasi-simple over $F$. Then the following conditions are equivalent.
\begin{enumerate}
\item $(\bG, \{\mu\})$ is fully Hodge-Newton decomposable and the basic conjugacy class $[b_{\mu,\mathrm{baisc}}]$ is superbasic.
\item The Tits datum associated with $(\bG, \{\mu\})$ is of extended Lubin-Tate type.
\end{enumerate}

\end{proposition}
\begin{proof}
$(2) \implies (1)$ is obvious. Let us prove $(1) \implies (2)$.

We may assume that $\bG$ is adjoint. Following \cite[\S3.4]{GHN2}, we write $\bG = \bG_1\times\bG_2\times\cdots\times \bG_r$ such that $\s$ induces an isomorphism from $\bG_i$ to $\bG_{i+1}$, $i=1,2,\ldots,r-1$ and from $\bG_{r}$ to $\bG_1$. We also write $\mu=(\mu_1,\mu_2,\ldots,\mu_r)$. Let $\bG'$ be the simple reductive group over $ F$ such that $\bG'(\breve F) = \bG_r(\breve F)$ and the Frobenius morphism on $\bG'(\breve F)$ is given by $\s'=\s^r\vert_{\bG_r}$.

We have the isomorphism of posets
$$B(\bG,\{\mu\}) \tilde\longrightarrow B(\bG',\{\mu'\})$$
sending $[(b_1,b_2,\ldots,b_r)]$ to $[b_r\s(b_{r-1})\cdots \s^{r-1}(b_1)]$. Here 
$$\mu' = \mu_r+\s_0(\mu_{r-1})+\cdots+\s_0^{r-1}(\mu_1).$$ 
Moreover, we have a natural bijection $\tS/\<\s\> \cong \tS_r/\<\s^{r}\>$. Then the assumption implies that $(\bG',\mu')$ is fully Hodge-Newton decomposable and $[b_{\mu',\indec}]$ is superbasic. Tracing the classification table in \cite[Theorem 3.5 and \S3.6]{GHN2}, we deduce that the Tits datum associated with $(\bG', \{\mu'\} )$ is $({A}_n,\id,\{\o_1^{\vee}\})$. Therefore, the Tits datum associated with $(\bG, \{\mu\})$ is of extended Lubin-Tate type. This completes the proof.
\end{proof}

\subsection{Essential gap between $[b_{\mu,\indec}]$ and $[b_{\mu,\max}]$}
In this subsection, we assume that $\bG$ is adjoint. Let $\bH$ and $\g$ be as in $\S\ref{sec:1.2}$. Recall that
$$\text{Ess-gap}([b_{\mu,\indec}],[b_{\mu,\max}]) = \mathbf{i}([b_{\mu,\indec}],[b_{\mu,\max}])+\mathbf{b_1}([b_{\mu,\indec}],[b_{\mu,\max}]),$$
where $\mathbf{i}([b_{\mu,\indec}],[b_{\mu,\max}])$ is the number of ``interior lattice points'' between $\nu_{\mu,\indec}$ and $\nu_{\mu,\max}$, and $\mathbf{b_1}([b_{\mu,\indec}],[b_{\mu,\max}])$ is the number of ``lattice points'' on $\nu_{\mu,\indec}$ that do not touch $\nu_{\mu,\max}$ (see \S \ref{sec:2.3} for the exact definition). 

\begin{lemma}\label{lem:ibibm=0}
We have $\mathbf{i}([b_{\mu,\indec}],[b_{\mu,\max}])=0$.
\end{lemma}

\begin{proof}
We may assume $\bG$ is quasi simple over $F$. If $\mu$ is central, then the statement is trivial. Assume $\mu$ is non-central. Let $\l\in X_*(T )$ be an arbitrary dominant coweight such that $\k(\l) = \k(\mu)\cdot\k(\g)$ and $\l^{\diamond}\ge \mu^{\diamond}$. Let $o\in \BS/\<\s_0\>$. By the construction of $[b_{\mu,\indec}]$ and $[b_{\mu,\max}]$ (cf. \cite[\S3.5]{HN2018}, \cite[\S4.1]{HNY}), we have 
$$ \<  \underline\mu ,\o_o\>\ge \<  \nu_{\mu,\max},\o_o\>\ge \<  \nu_{\mu,\indec},\o_o\>\ge \max\{ a\mid a\in  \<\underline\l,\o_o\>+\BZ,  a< \<\mu,\o_o\> \}.$$
%and
%$$  \<\underline  \mu ,\o_o\>  > \<  \nu_{\mu,\indec},\o_o\>  .$$

%\remind{$  \<  \mu ,\o_o\>  > \<  \nu_{\mu,\indec},\o_o\>$ can not be deduced from the above $\ge$,$\ge$,$\ge$}

Then one can easily deduce that
\begin{align*}
\lceil \<\underline\l  - \nu_{\mu,\indec} ,\o_o\>\rceil - \lceil\<\underline\l   - \nu_{\mu,\max} ,\o_o\> \rceil =\begin{cases}
    1 & \text{ if }\<\underline\l-\nu_{\mu,\max},\o_o\>\in\BZ \text{ and } \< \nu_{\mu,\max},\o_o\>> \< \nu_{\mu,\indec},\o_o\>;\\
    0 &\text{otherwise}.
\end{cases}
\end{align*}
Recall that
$$\mathbf {b_2}([b_{\mu,\indec}],[b_{\mu,\max}]) = \sharp \{o\in\BS/\<\s_0\> \mid \<\underline\l-\nu_{\mu,\max},\o_o\>\in\BZ, \< \nu_{\mu,\max},\o_o\>> \< \nu_{\mu,\indec},\o_o\> \}  .$$
It follows that
$$\mathbf{i}([b_{\mu,\indec}],[b_{\mu,\max}]) = \sum_{o\in\BS/\<\s_0\>} \bigl(\lceil \<\underline\l  - \nu_{\mu,\indec} ,\o_o\>\rceil - \lceil\<\underline\l   - \nu_{\mu,\max} ,\o_o\> \rceil\bigr) - \mathbf{b_2}([b_{\mu,\indec}],[b_{\mu,\max}]) =0.$$
\end{proof}

\begin{corollary}\label{lem:nuinum}
The following conditions are equivalent.
\begin{enumerate}%[(1)]
    \item  $\text{Ess-gap}([b_{\mu,\indec}],[b_{\mu,\max}]) = 0$.
    \item  $\mathbf {b_1}(b_{\mu,\indec},b_{\mu,\max}) = 0$.
    \item For every $o\in \BS/\<\s_0\>$ such that $\<  \nu_{\mu,\indec},\o_o\> =\<  \l([b_{\mu,\indec}]),\o_o\> $, one has $\<\nu_{\mu,\max}-\nu_{\mu,\indec},\o_o\> = 0$. In other words, $\nu_{\mu,\max}$ touches every ``lattice point'' of $\nu_{\mu,\indec}$.
    
\end{enumerate}

\end{corollary}   

\begin{proof}
The equivalence of (1) and (2) follows from Proposition \ref{prop:ige0} and Lemma \ref{lem:ibibm=0}. The equivalence of (2) and (3) follows from the definition of $\mathbf {b_1}([b_{\mu,\indec}],[b_{\mu,\max}]) $.\qedhere
\end{proof}

\subsection{Main result} For any $[b]\in B(\bG,\{\mu\})$, there is a unique minimal $\s_0$-stable subset $J\subseteq \BS$ such that $(\mu,[b])$ is Hodge-Newton decomposable with respect to $M_J$. In fact, $J$ equals the union of $I(\nu_b)=\{i\in\BS\mid \<\nu_b,\a_i\> = 0\}$ and the set $\{i\in\BS\mid \<\mu^{\diamond} - \nu_b,\o_i\>\ne 0 \}$.

%For any $\s_0$-stable subset of $\BS$, let $B(\bG,\{\mu\})_{J-\mathrm{decom}}\subset B(\bG,\{\mu\}) $ be the subset of elements $[b]$ such that $M_J$ is the minimal standard $\s_0$-stable Levi subgroup such that $(\mu,[b])$ is Hodge-Newton decomposable with respect to $M_J$. Then we have
%$$B(\bG,\{\mu\}) = \bigsqcup_{J\subset \BS,\s_0(J)=J}B(\bG,\{\mu\})_{J-\mathrm{decom}}.$$
%By definition, $B(\bG,\{\mu\})_{\BS-\mathrm{decom}} $ is the set of Hodge-Newton indecomposable element. Note that each $B(\bG,\{\mu\})_{J-\mathrm{decom}}$ is not necessarily nonempty.

We now state the main Theorem of this paper.
\begin{theorem}\label{thm:main}
Let $[b] \in B(\bG,\{\mu\})$. Let $J$ be the minimal $\s_0$-stable subset of $\BS$ such that $(\mu,[b])$ is Hodge-Newton decomposable with respect to $M_J$. Let $\mathfrak{P}_J^{\s}$ be the set of $\s$-stable parabolic subgroups over $\breve F$ that conjugate to $P_J$ and contain $T$. Let $K$ be a $\s$-stable subset of $\tS$ such that $W_K$ is finite. Then the following statements are equivalent.
\begin{enumerate}%[(1)]
\item $\dim X(\mu,b)_K = 0$.

\item For some (or equivalently, any) $P\in \mathfrak{P}_J^{\s}$ with semi-standard Levi decomposition $P=MN$, the Tits datum associated with $( M , \{\mu_{P}\})$ is of extended Lubin-Tate type. Here, $\{\mu_{P}\}$ is as in \S\ref{sec:HN}.

\item $[b_{\mu,\max}]$ is $\mu$-ordinary and $\text{Ess-gap}([b],[b_{\mu,\max}])=0$.

\end{enumerate}
\end{theorem}

\begin{proof}[Proof of Theorem \ref{thm:main}]
The equivalence of ``some" and ``any" in (2) follows from the commutative diagram (\ref{eq:3.1}) and the definition of $\mu_{P}$.
    
$(1) \implies (3)$ follows directly from Proposition \ref{prop:dimxmub}. 

$(2) \implies (1)$ follows directly from Theorem \ref{thm:HN-decom}.

Let us prove $(3) \implies (2)$. We may assume that $\bG$ is adjoint. For any $P\in \mathfrak{P}_J^{\s}$ with semi-standard Levi decomposition $P=MN$, by Proposition \ref{prop:levi}, the condition (3) holds for the Levi subgroup $M$. Hence we may assume that $J = \BS$ and $(\mu,[b])$ is Hodge-Newton indecomposable. 

Since $[b]\le [b_{\mu,\indec}]$, by the additive property of the essential gap (see \S\ref{sec:2.1}), we have $\text{Ess-gap}([b_{\mu,\indec}],[b_{\mu,\max}])=0$. By the definition of extended Lubin-Tate type, we may assume that $\bG$ is quasi-simple over $F$ and $\mu$ is non-central. Since $\nu_{\mu,\max} =\mu^{\diamond}$, using Corollary \ref{lem:nuinum}, we obtain that $\nu_{\mu,\indec}$ does not have ``lattice points", in other words, for any $o\in\BS/\<\s_0\>$, we have $\< \nu_{\mu,\indec},\o_o\> \ne  \< \l ( [b_{\mu,\indec}]),\o_o\>$. By (\ref{eq:1.3}), this implies that $[b_{\mu,\indec}]$ is superbasic. In particular, we have $[b_{\mu,\indec}] = [b] = [b_{\mu,\mathrm{basic}}]$ and $(\bG, \{\mu\})$ is fully Hodge-Newton decomposable. The statement follows from Proposition \ref{prop:nu0superbasic}.
\end{proof}

\begin{remark}
\begin{enumerate}%[(1)]
%    \item The equivalence of conditions (1) and (2) for basic $[b]$ is proved in \cite[Theorem 4.1]{GHR}. In fact, in the proof of Theorem \ref{thm:main} $(3) \implies (2)$, when we deduce that $[b]$ is basic, we can apply the result of \cite{GHR} for the basic case. Here we give a conceptual proof independently, utilizing the fact that $(\bG, \{\mu\})$ is fully Hodge-Newton decomposable and $[b]$ is superbasic.
    
    \item The dimension of $X(\mu,b)_K$ depends on the level $K$. However, conditions (2) (3) of Theorem \ref{thm:main} do not.

    \item The equivalence $(1) \iff (2)$ for the basic $\s$-conjugacy class $[b]$ is proved in \cite[Theorem 4.1]{GHR}. Our proof is independent of the proof in \cite{GHR}.
    
    \item Using condition (3) of Theorem \ref{thm:main}, one directly sees that the set $\{[b]\mid\dim X(\mu,b)=0\}$ is saturated. That is, if $\dim X(\mu, b) = 0$, then $\dim X(\mu, b') = 0$ for all $[b'] \in B(\bG,\{\mu\})$ such that $[b] \le [b']$. However, the set $\{[b]\mid \dim X(\mu,b)=0\}$ does not have a unique minimum in general. For example, assume that $\bG=\GL_4$ and $\mu=\o_2^{\vee}$. We have $\{[b]\mid \dim X(\mu,b)=0\} = \{(1,1,0,0),(1,\frac{1}{2},\frac{1}{2},0),(1,\frac{1}{3},\frac{1}{3},\frac{1}{3}),(\frac{2}{3},\frac{2}{3},\frac{2}{3},0)\} $, which does not have a unique minimum.

    \item In the situation of condition (2) of Theorem \ref{thm:main}, using Theorem \ref{thm:HN-decom} and \cite[Proposition 5.3]{He14}, we get 
    \begin{align*}
        X(\mu,b) \cong \bigcup_{P =MN\in \mathfrak{P}_J^{\s}}X_{\t_M}^M(b_{P})\cong\bigcup_{P =MN\in \mathfrak{P}_J^{\s}}J_{ \t_M }(F)/J_{ \t_M }(F)\cap\breve\CI \cong\bigcup_{P =MN\in \mathfrak{P}_J^{\s}}  \Omega_M^{\s},
    \end{align*}
    where $\t_M$ is the length zero element in $\Adm^{M}(\{\mu_{P}\})$, $J_{\t_M}(F)$ is the $\s$-centralizer of $\dot{\t_M}$ in $M(\breve F)$ and $\Omega_M^{\s}$ is the set $\s$-invariant length-zero elements in $\tW_M$. This result can be generalized to the general parahoric level $K$ using EKOR strata. 
\end{enumerate}
\end{remark}

\begin{corollary}
There exists some $[b]\in B(\bG)$ with $\dim X(\mu,b)=0$ if and only if $[b_{\mu,\max}]$ is $\mu$-ordinary.
\end{corollary}

\begin{proof}
   The ``if" direction follows from Theorem \ref{thm:main}. For the ``only if" direction, observe that $\dim X(\mu,b_{\mu,\max}) = \<\underline\mu-\nu_{\mu,\max},2\rho\> =  \<\underline\mu-\mu^{\diamond},2\rho\>  = 0$. \qedhere
\end{proof}

Let $P\in\mathfrak{P}_J^{\s}$ and let $z_{P}$ be as in \S\ref{sec:HN}. Define $\s_J  := \mathrm{Ad}(z_{P})^{-1} \circ \s \circ \mathrm{Ad}(z_{P})$. Then $M_J$ is $\s_J$-stable. Let $\tilde{\Sigma}_J$ be the absolute local Dynkin diagram associated with $M_J$ and let $\d_J$ be the diagram automorphism on $\tilde{\Sigma}_J$ induced by $\s_J$. By the commutative diagram \ref{eq:3.1}, the Tits datum $(\tilde{\Sigma}_J, \d_J, \{\mu\})$ is independent (up to automorphism) of the choice of $P\in\mathfrak{P}_J^{\s}$. As a consequence, we can restate Condition (2) of Theorem \ref{thm:main} as

$(2')$ The Tits datum $(\tilde{\Sigma}_J, \d_J, \{\mu\})$ is of extended Lubin-Tate type.

Furthermore, by Lemma \ref{lem:tauJ}, we see that $\s_J$ equals $\Ad(\t_J)\circ \s_0$ as an action on $\tW_J$.

%Let $\tilde b = z_{P}^{-1} b_P  \s(z_{P }) $ and let $[\tilde b]_{\BM_J}$ be its $M_J(\breve F)$-$\s_J$-conjugacy class.

%Let $P'$ be another element in $\mathfrak{P}_J^{\s}$. Define $\s_J'$ and $\BM_J'$ similarly.
%We claim that the triple $(\BM_J,\{\mu\})$ and $(\BM_J',\{\mu\})$ are isomorphic under conjugation by some length zero element of $\tW_J$. Indeed,
%Let $u\in \tW^{\s}$ such that ${}^{u}P = P'$ and ${}^{u}  \breve {\CI}_{M}   =  \breve {\CI}_{M'}$ (see \S\ref{sec:HN}). Then $uz_{P} = z_{P'}x$ for some $x\in {\tW}_J = X_*(T)_{\G_0}\rtimes W_J$. Then we have
%$${}^{x}\breve \CI_{M_J}={}^{ z_{P'}^{-1}uz_{P}}\breve \CI_{M_J} = {}^{z_{P'}^{-1}u}\breve \CI_{M} = {}^{z_{P'}^{-1}}\breve \CI_{M'}=\breve \CI_{M_J}.$$ This means that $x$ is a length zero element in $\tW_J$. We have the following commutative diagram, which is a refinement of the commutative diagram in \S\ref{sec:HN}.
%\begin{align*}
%\begin{tikzcd}[ampersand replacement=\&]
%\&M(\breve F)\ar[d,"\s"]  \&M_J(\breve F) \ar[l,"\Ad(z_{\mathbf {P}})"']\ar[d,"\s_J"] \ar[r,"\Ad(x)"] \&M_J(\breve F) %\ar[d,"\s_J'"]\ar[r,"\Ad(z_{P'})"]\& M' (\breve F)\ar[d,"\s"]\\
%\&M(\breve F) \&M_J(\breve F) \ar[l,"\Ad(z_{\mathbf {P}})"']\ar[r,"\Ad(x)"] \&M_J(\breve F)\ar[r,"\Ad(z_{P'})" ]\& M' (\breve F)
%\end{tikzcd}
%\end{align*}

\subsection{Examples}\label{sec:ex}
If $[b_{\mu,\max}]$ is $\mu$-ordinary, then all three conditions in Theorem \ref{thm:main} trivially hold for $[b] = [b_{\mu,\max}]$. In this subsection, We give some examples where $[b]\ne[b_{\mu,\max}]$, $[b]\ne [b_{\mu,\mathrm{basic}}]$ and $\dim X(\mu,b) = 0$.% In the rest of this paper, we use Bourbaki's notations \cite{Bourbaki}. 

\subsubsection{Fully Hodge-Newton decomposable case}
The most typical example is when $(\bG,\{\mu\})$ is fully Hodge-Newton decomposable and $[b]$ is non-basic. In this case, we have $\dim X(\mu,b)=0$ by \cite[Theorem 3.3 (3)]{GHN2}. See the Example \ref{ex:5.8}$-$\ref{ex:5.11} below.

\begin{example}\label{ex:5.8}
Assume that the Tits datum associated with $(\bG,\{\mu\})$ is $({A}_3, \z_0,  \{\o_2^{\vee}\})$, where $\z_0$ is the nontrivial diagram automorphism. In this case $\nu_{\mu,\max} = (\frac{1}{ 2},\frac{1}{ 2},-\frac{1}{ 2},-\frac{1}{ 2}) = \mu^{\diamond}$. We have 
$$B(\bG,\{\mu\}) = \{[b_{\mu,\max}], [b_1], [b_{\mu, \mathrm{basic}}]\} ,$$
where $\nu_{b_1} = (\frac{1}{2},0,0,-\frac{1}{2})$. We have $X(\mu,b_1) = 0$. In this case, $J = \{s_2\}$ and one can directly check that $\s_J=\s$ is identity on $\tW_J$ and the Tits datum $(\tilde{\Sigma}_J, \d_J, \{\mu\})$ is $({A}_1,\id,\{\o_1^{\vee}\})$. Moreover, we have $\text{Ess-gap}([b_1],[b_{\mu,\max}]) = 0.$
\end{example}
\begin{example}\label{ex:5.9}
Assume that the Tits datum associated with $(\bG,\{\mu\})$ is $({A}_4,\z_0, \{\o_1^{\vee}\})$, where $\z_0$ is the nontrivial diagram automorphism. In this case $\nu_{\max} = (\frac{1}{ 2},0,0,0,-\frac{1}{ 2}) = \mu^{\diamond}$. We have 
$$B(\bG,\{\mu\}) = \{[b_{\mu,\max}], [b_1], [b_{\mu,\mathrm{basic}}]\} ,$$
where $\nu_{b_1} = (\frac{1}{4},\frac{1}{4},0,-\frac{1}{4},-\frac{1}{4}) $. We have $X(\mu,b_{1}) = 0$. In this case, we have $J = \{s_1,s_4\}$. One can directly check that the Tits datum $(\tilde{\Sigma}_J, \d_J, \{\mu\})$ is $(\mathrm{Res}_{F_2/F}({A}_1,\id),\{(\o^{\vee}_1,0)\})$. Moreover, we have $\text{Ess-gap}([b_1],[b_{\mu,\max}]) = 0$.
\end{example}

\begin{example}\label{ex:5.10}
Assume that the Tits datum associated with $(\bG,\{\mu\})$ is $({A}_5,\Ad(\t_1)\circ \z_0,  \{\o_1^{\vee}\})$, where $\z_0$ is the nontrivial diagram automorphism. In this case $\nu_{\max} = (\frac{1}{ 2},0,0,0,0,-\frac{1}{ 2}) = \mu^{\diamond}$. We have 
$$B(\bG,\{\mu\}) = \{[b_{\mu,\max}], [b_1], [b_{\mu,\mathrm{basic}}]\} ,$$
where $\nu_{b_1} = (\frac{1}{4},\frac{1}{4},0,0,-\frac{1}{4},-\frac{1}{4}) $. Then $X(\mu,b_{1}) = 0$. In this case, we have $J = \{s_1,s_3,s_5\}$. One can directly check that the Tits datum $(\tilde{\Sigma}_J, \d_J, \{\mu\})$ is $(\mathrm{Res}_{F_2/F}({A}_1,\id),\{(\o^{\vee}_1,0)\}) \times ({A}_1,\id,0)$. Moreover, we have $\text{Ess-gap}([b_1],[b_{\mu,\max}]) = 0.$
\end{example}

\begin{example}\label{ex:5.11}
Assume that the Tits datum associated with $(\bG,\{\mu\})$ is $({D}_n,\Ad(\t_1), \{\o_1^{\vee}\})$. In this case $\nu_{\max} = (1,0,0,\ldots,0 ) = \mu $. We have 
$$B(\bG,\{\mu\}) = \{[b_{\mu,\max}], [b_2],[b_3],\ldots,[b_{n-2}], [b_{\mu,\mathrm{basic}}]\} ,$$
where $\nu_{b_i} = \bigl((\frac{1}{i})^{(i)},(0)^{(n-i)} \bigr) $. Here the notation $(a)^{(k)} $ means $(a,a,\ldots,a)$ of length $k$. 
Then $X(\mu,b_{i}) = 0$ for $i = 2,3,\ldots,n-2$. In this case, one can directly check that the Tits datum $(\tilde{\Sigma}_J, \d_J, \{\mu\})$ is $ ({A}_{i-1},\id,\{\o_1^{\vee}\}) \times ({A}_{n-i},\id,\{0\} ) $. Moreover, we have $\text{Ess-gap}([b_i],[b_{\mu,\max}]) = 0$.
\end{example}
\begin{remark}
Example 4.11 is missing in the classification list of fully Hodge-Newton decomposable cases in \cite[Theorem 3.5]{GHN2}.
\end{remark}

\subsubsection{The $\GL_n$ case}
Consider the case $\bG=\GL_n$. Then $[b_{\mu,\max}] $ is $\mu$-ordinary. We identify $X_*(T)$ with $\BZ^{n}$. Write $\mu = (a_1,a_2,\ldots,a_n)$. By the explicit description in \S \ref{sec:2.2}, there is $[b]<[b_{\mu,\max}]$ such that $\text{Ess-gap}([b],[b_{\mu,\max}])=0$ if and only if $a_i-a_{i+1}=1$ for some $i$. For example, we take $n = 5$, $\mu = (2,1,0,-1,-1)$ and $\nu_b = (\frac{3}{2},\frac{3}{2},-\frac{2}{3},-\frac{2}{3},-\frac{2}{3})$. Then we can easily see that $\mathbf i = \mathbf{b_1}= 0$ and the Tits datum $(\tilde{\Sigma}_J, \d_J, \{\mu\})$ is $ ({A}_{1},\id,\{\o_1^{\vee}\}) \times ({A}_{2},\id,\{\o_1^{\vee}\}) $.

\subsection{The $\mu$-ordinary condition}\label{sec:mu-ordinary}
Recall that the action of $\s$ on the apartment $\CA$ induces an affine transformation on $V$ (which depends on the choice of the special vertex of the base alcove $\mathfrak{a}$, see \S\ref{sec:1.1}). Then either $\s(0)=0$ or $\s(0)$ equals some (non-zero) minuscule coweight in $X_*(T_{\ad})_{\G_0}\subseteq V$ (note that $X_*(T_{\ad})_{\G_0}$ is torsion free). If $\bG$ is not adjoint, then $\s(0)$ may not be lifted to an element in $X_*(T)_{\G_0}$.

Recall that for any $v\in V$, we denote $I(v) = \{i \in \BS\mid \<v,\a_i\>=0\} $. By \cite[Corollary 3.6]{HN2018}, the generic $\s$-conjugacy class $[b_{\mu,\max}]$ is $\mu$-ordinary if and only if $\< \s(0) ,\o_o\>\in \BZ$ for any $o \in \bigl(\BS - I(\mu^{\diamond} )\bigr)/\<\s_0\>$ (We already use this criterion in the proof of Proposition \ref{prop:levi} (1)). Using this criterion and the explicit computation in \cite{Bourbaki}, one can easily classify the Coxeter data of all non-quasi-split groups such that $[b_{\mu,\max}]$ is $\mu-$ordinary. We use Bourbaki's notations \cite{Bourbaki}. Note that the type $\tilde{A}$ case is already established in \cite[Example~2.5]{HN17}.

\begin{proposition}\label{prop:table}
Assume that $\bG$ is non quasi-split such that $[b_{\mu,\max}]$ is $\mu$-ordinary. Then the associated Coxeter datum up to diagram automorphism is one of the following.

\begin{center}
\begin{tabular}{ |c|c|c|c|c|c|c|c|c| } 
 \hline
 Coxeter datum &  Condition on $I(\mu^{\diamond})$\\ 
 \hline
 $ (\tilde{A}_{n-1}, \Ad(\t_i), \{\mu\})$, $1 \le i \le \frac{n}{2}$ &  $\BS-I(\mu)\subseteq \{r,2r,\ldots,(d-1)r\}$, where $d = \mathrm{g.c.d}(i,n)$, $r = \frac{n}{d}$ \\ 
 \hline
 $(\tilde{A}_{n-1},  \Ad(\t_1)\circ \zeta_0, \{\mu\})$, $n$ even &  $\frac{n}{2}\in I(\mu^{\diamond})$ \\ 
  \hline
 $(\tilde{B}_{n}, \Ad(\t_1), \{\mu\})$, $n \ge 2$   &  $n\in I(\mu)$\\ 
        \hline
 $(\tilde {C}_{n}, \Ad(\t_n), \{\mu\})$, $n \ge 2$  & $\{1,3,5,\ldots\}\subseteq I(\mu)$ \\ 
  \hline
 $(\tilde {D}_{n},  \Ad(\t_1) ,\{\mu\})$, $n\ge5$ &   $\{n-1,n\} \subseteq I(\mu) $ \\ 
  \hline
 $(\tilde {D}_{n}, \Ad(\t_n) ,\{\mu\})$, $n$ odd &  $\BS-I(\mu)\subseteq \{2,4,6,\ldots,n-3\}$ \\ 
  \hline
 $(\tilde {D}_{n}, \Ad(\t_n) ,\{\mu\})$, $n=4k+2$ & $\BS-I(\mu)\subseteq \{2,4,6,\ldots,n-2,n-1\}$ \\ 
  \hline
 $(\tilde {D}_{n}, \Ad(\t_n) ,\{\mu\})$, $n=4k $ & $\BS-I(\mu)\subseteq \{2,4,6,\ldots,n-2,n\}$\\ 
  \hline
 $(\tilde {D}_{n}, \Ad(\t_n)\circ\zeta_0 ,\{\mu\})$, $n$ odd &  $\BS-I(\mu^{\diamond})\subseteq \{2,4,6,\ldots,n-3,n-1,n\}$\\ 
  \hline
 $(\tilde {D}_{n}, \Ad(\t_n)\circ\zeta_0 ,\{\mu\})$, $n$ even & $\BS-I(\mu^{\diamond})\subseteq \{2,4,6,\ldots,n-4,n-2\}$\\ 
  \hline 
 $(\tilde {E}_{6}, \Ad(\t_1) ,\{\mu\})$ & $\BS-I(\mu)\subseteq \{2,4   \}$  \\
  \hline
   $(\tilde {E}_{7},\Ad(\t_7) ,\{\mu\})$&  $\BS-I(\mu)\subseteq \{1,3,4,6   \}$  \\
  \hline
\end{tabular}
\end{center}
Here, for type $A$ and type $D$, $\zeta_0$ is the nontrivial order 2 diagram automorphism. 
%Here $\zeta_0$ stands for the nontrivial diagram automorphism for type $\tilde{A}$ and the diagram automorphism permuting $s_{n-1}$ and $s_n$ and preserving the others for type $\tilde{D}$. 
\end{proposition}

\subsection{The $({A}_{n-1},\Ad(\t_i),\{\mu\})$ case}
In this subsection, we prove the following

\begin{proposition}\label{prop:nob} Assume that the Tits datum associated with $(\bG,\{\mu\})$ is $({A}_{n-1},\Ad(\t_i),\{\mu\})$ for some $i\ne0$. Assume that $[b_{\mu,\max}]$ is $\mu$-ordinary. Then there is no $[b]<[b_{\mu,\max}]$ such that $\text{Ess-gap}([b],[b_{\mu,\max}])=0$.
\end{proposition}

\begin{proof}
It suffices to consider the group $\GL_n(\breve F)$ with the usual Frobenius action twisted by $\Ad(\t_i)$. Denote $\nu_\t = (\frac{i}{n},\frac{i}{n},\ldots,\frac{i}{n})$. Then $\s(0) = \o_{i}^{\vee} = \left((1-\frac{i}{n})^{(i)},(-\frac{i}{n})^{(n-i)}\right)$. For any vector $\l\in \BQ^{n}$, set $\tilde \l = \l + \nu_\t$. We consider the Newton polygon corresponding to $\widetilde {\nu_b} $ instead of $\nu_b$ since $\widetilde{\nu_b}$ has integral coordinates vertices but $\nu_b$ does not. Let $[b_1],[b_2] $ be two $\s$-conjugacy classes with Newton points $\nu_{ 1}, \nu_{2}$ such that $[b_1]\le [b_2]$. Let $\widetilde{P_{12}}$ be the region between $\widetilde{\nu_{ 1}}$ and $\widetilde{\nu_{ 2}}$. Set 
$$\text{Ess-gap}_{[0,n]}(\widetilde{\nu_1},\widetilde{\nu_2}) := \mathbf i + \mathbf {b_1},$$ 
where $\mathbf {b_1}$ is the number of lattice points on the Newton polygon of $\widetilde{\nu_1}$ but not on the Newton polygon of $\widetilde{\nu_2}$, and $\mathbf {i}$ is the number of lattice points in the interior of $\widetilde{P_{12}}$. Then we have $\text{Ess-gap}([b_1],[b_2]) = \text{Ess-gap}_{[0,n]}(\widetilde{\nu_1},\widetilde{\nu_2})$.

Let $d=\mathrm{g.c.d}(i,n)$ and $r=\frac{n}{d}$. Applying Proposition \ref{prop:table}, we see that $\mu$ is of the form $\mu = \left((a_1)^{(r)},(a_2)^{(r)},\ldots,(a_d)^{(r)}\right)$ with $a_1\ge a_2\ge \cdots \ge a_d$. Then
$$\tilde{\mu} = \left(\left(a_1+\frac{i}{n}\right)^{(r)},\left(a_2+\frac{i}{n}\right)^{(r)},\ldots,\left(a_d+\frac{i}{n}\right)^{(r)}  \right).$$
Since $r$ is the minimal integer such that $n$ divides $ir$, we have
\begin{align*}\tag{a}
\<\tilde{\mu},\o_j \>\in\BZ \text{ if and only if } j\in \{r,2r,\ldots,(d-1)r\},    
\end{align*}
where $\o_{j}=((1)^{(j)},(0)^{(n-j)})$ is the fundamental weight.

Let $k\in \{1,2,\ldots,d-1\}$ such that $a_k>a_{k+1}$. Define $\nu_k =\left((c_1)^{(r)},(c_2)^{(r)},\ldots,(c_d)^{(r)}\right)$, such that $c_i = a_i$ for $i\notin\{ k,k+1\}$ and $c_k=c_{k+1} = \frac{1}{2}(a_k+a_{k+1})$. Then $\nu_k<\mu$. Direct computation shows that 
$$\<\widetilde{\mu}-\widetilde{\nu_k},\o_{rk}\> = \<\widetilde{\mu}-\widetilde{\nu_k},\o_{rk}\> = \frac{r(a_k-a_{k+1})}{2} \ge 1.$$
Using (a), we get $\text{Ess-gap}_{[0,n]}(\widetilde{\nu_k},\widetilde{\mu})\ge1$. 

Let $\nu$ be a maximal element such that $\nu<\mu$ and $\text{Ess-gap}_{[0,n]}(\widetilde{\nu},\widetilde{\mu})=0$. Then $\nu>\nu_k$ for some $k$ with $a_k>a_{k+1}$. Therefore, $\widetilde{\nu}$ must contain some lattice point that lies between $\widetilde{\nu_k}$ and $\widetilde{\mu}$. This contradicts to $\text{Ess-gap}_{[0,n]}(\widetilde{\nu},\widetilde{\mu})=0$. This completes the proof.
\end{proof}

\printbibliography

\end{document}